\documentclass[9pt]{article}
\usepackage{graphicx}

\title{Periods of $E$-operators}

\usepackage[utf8]{inputenc}
\usepackage{geometry}
\date{\vspace{-5ex}}

\usepackage{amsthm}

\newtheorem{theorem}{Theorem}[section]
\newtheorem{corollary}[theorem]{Corollary}
\newtheorem{lemma}[theorem]{Lemma}
\newtheorem{proposition}[theorem]{Proposition}

\theoremstyle{definition}
\newtheorem{definition}[theorem]{Definition}

\theoremstyle{remark}
\newtheorem{remark}[theorem]{Remark}
\newtheorem{aside}[theorem]{Aside}

\newtheorem{assumption}[theorem]{Assumption}

\usepackage{amsmath}
\usepackage{amssymb}
\usepackage{amsfonts}
\usepackage{tensor}
\usepackage{centernot}
\usepackage{mathtools}
\usepackage{xcolor}
\usepackage{hyperref}
\usepackage{bm}
\usepackage{array}
\hypersetup{hidelinks}
\usepackage{doi} 
\usepackage{graphicx}
\usepackage[shortlabels]{enumitem}
\usepackage{tikz-cd}
\usepackage{adjustbox}
\usepackage{wrapfig}
\usepackage{chngcntr}
\usepackage{accents}
\usepackage{simpler-wick}
\usepackage{cleveref}
\usepackage{mathrsfs}
\usepackage{calligra}
\usepackage{tikz}
\usepackage{pgfplots}
\usepackage[dvipsnames]{xcolor}

\renewcommand{\bar}[1]{\accentset{\rule{.4em}{.4pt}}{#1}}

\newcommand{\bb}[1]{\mathbb{#1}}
\newcommand{\mbf}[1]{\mathbf{#1}}
\newcommand{\mfr}[1]{\mathfrak{#1}}
\newcommand{\mcal}[1]{\mathcal{#1}}
\newcommand{\mscr}[1]{\mathscr{#1}}
\newcommand{\spec}[1]{\text{Spec}\,#1}

\newcommand{\abs}[1]{\vert #1 \vert}

\newcommand{\id}[0]{\text{id}}

\newcommand{\modd}[2]{\text{Mod}_{#1}(#2)}
\newcommand{\rob}[0]{\widetilde{\bb{G}_m}}

\newcommand{\D}[1]{\mscr{D}_{#1}}

\newcommand{\qbar}[0]{\overline{\bb{Q}}}
\newcommand{\qu}[0]{\mathbf{A}^{qu}}

\numberwithin{equation}{section}

\newcommand*\diff{\mathop{}\!\mathrm{d}}

\usepackage[T1]{fontenc}
\DeclareFontFamily{T1}{calligra}{}
\DeclareFontShape{T1}{calligra}{m}{n}{<->s*[1.3]callig15}{}
\DeclareMathAlphabet\mathcalligra   {T1}{calligra} {m} {n}
\DeclareMathAlphabet\mathzapf       {T1}{pzc} {mb} {it}
\DeclareMathAlphabet\mathchorus     {T1}{qzc} {m} {n}
\DeclareMathAlphabet\mathrsfso      {U}{rsfso}{m}{n}

\makeatletter
\newcommand{\address}[1]{\gdef\@address{#1}}
\newcommand{\email}[1]{\gdef\@email{\url{#1}}}
\newcommand{\@endstuff}{\par\vspace{\baselineskip}\noindent\small
\begin{tabular}{@{}l}\scshape\@address\\\textit{E-mail address:} \@email\end{tabular}}
\AtEndDocument{\@endstuff}
\makeatother

\author{Ben Snodgrass} 
\address{Ben Snodgrass, Math. Institut, Universit\"at Freiburg, Ernst-Zermelo-Str. 1, 79104 Freiburg, Germany}
\email{benedict.snodgrass@math.uni-freiburg.de}

\begin{document}

\maketitle

\begin{abstract}
    We define an enlargement of the class of exponential periods by applying the rapid decay cohomology theory of Hien to a larger class of integrable connections on varieties than those twisted by a regular function. The main and motivating examples are given by those associated to $E$-operators. This delivers a method to realise a class of absolutely convergent integrals involving $E$-functions as matrix elements of a period pairing, which we call \emph{$E$-periods}. Furthermore, we generalise rapid decay cohomology to singular varieties and prove a version of Nori's basic lemma in this setting.
\end{abstract}

\tableofcontents

\section{Introduction}

In this paper, we introduce a new subset of the complex numbers which can be realised as coefficients of a period pairing between the algebraic de Rham cohomology and the rapid decay cohomology of a class of integrable connection with irregular singularities. They contain the exponential periods of Kontsevich. We also prove a version of Nori's basic lemma in this new setting. \par

The following is a mildly simplified version of the main result. After Andr\'e \cite[Th\'eor\`eme 4.3]{andré00}, any $E$-operator has a basis of solutions of the form
\begin{equation} \label{eqn: andre basis 1}
        \begin{pmatrix}
        F_1 & \dots & F_n
    \end{pmatrix} z^{\Gamma_0},
\end{equation}
where $F_i$ are $E$-functions and $\Gamma_0$ is an upper triangular matrix with entries in $\bb Q$.

\begin{theorem}[See Thm. \ref{Thm: main result} and \S \ref{Sec: determining structure}] \label{Thm: intro}
    Let $X$ be a smooth variety over a number field $k$ equipped with a morphism ${f: X \to \bb G_m}$ and let $\mcal E$ be the connection on $\bb G_m$ associated to an $E$-operator satisfying a condition \nobreak (Assumption \ref{Assump: half plane}). Choose a basis as in \eqref{eqn: andre basis 1}. Then the rapid decay cohomology groups $H^i_{rd}(X, f^*\mcal E)$ inherit a structure over some cyclotomic extension $\bb Q(\zeta)$. In other words, there is a $\bb Q(\zeta)$-vector space $H^i_{rd}(X, f^*\mcal E)_{\bb Q(\zeta)}$ equipped with a canonical isomorphism
    \begin{equation}
        H^i_{rd}(X, f^*\mcal E) \simeq H^i_{rd}(X, f^*\mcal E)_{\bb Q(\zeta)} \otimes_{\bb Q(\zeta)} \bb C.
    \end{equation}
\end{theorem}
\noindent We also prove a version that is independent of the choice of basis \eqref{eqn: andre basis 1}, see Remark \ref{Rem: non canon}. \vspace{0.1cm}

Together with the period isomorphism of Hien \cite{Hien09}, this allows us to define \emph{$E$-periods} as the matrix elements of this isomorphism with respect to this $\bb Q(\zeta)$-structure, see Definition \ref{Defn: E period}. \par

\vspace{0.3cm}

Periods are complex numbers which describe the comparison isomorphism between algebraic de Rham cohomology of a variety $X$ over a number field $k$ and the singular cohomology of its analytification. This isomorphism can be interpreted as a pairing of algebraic differential forms on $X$ with cycles in its singular homology groups, given by integration. Their study goes back to Grothendieck, who formulated his famous period conjecture on the transcendence properties of periods. Beginning with a letter from Deligne to Malgrange \cite[p. 17]{DPM07} and a review of Kontsevich and Zagier \cite{KZ01}, so-called \emph{exponential periods} have come to the fore. Examples of exponential periods include $e^\alpha$ for any $\alpha \in \qbar^\times$, $\sqrt \pi$ and $\gamma$, the Euler constant, all of which are not expected to be classical periods. Both classical and exponential periods are described by certain `Galois groups', which can be constructed from the motivic theories of Nori \cite{HMS17} and Fres\`an and Jossen \cite{FJexp} respectively. \par
The main approach begins by calculating the cohomology groups of a twisted algebraic de Rham complex, with differential given by $\diff - \diff f$, where $f: X \to \bb{A}^1$ is a regular function on a smooth variety $X$. This is the de Rham complex of an integrable connection which we write as $\mcal{O}_X e^f$. Its cohomology groups are finite-dimensional $k$-vector spaces. Beginning with Bloch and \'Esnault \cite{blochesnault} and developed further by Hien and Roucairol \cite{HienRouc08}, a natural dual to this cohomology theory called \emph{rapid decay homology} $H_i^{rd}(X, f)_\bb{Q} \in \text{Vec}_\bb{Q}$ was developed. Roughly speaking, this is a kind of singular homology in which the cycles are permitted to approach infinity in regions where $e^{-f}$ decays exponentially. It can be defined via relative singular cohomology, and hence is naturally a $\bb{Q}$-vector space. Furthermore, there is a comparison isomorphism
\begin{equation} \label{eqn: PP exp}
    H^i_{dR}(X, f) \otimes_k \bb{C} \xrightarrow{\sim} H^i_{rd}(X, f)_{\bb{Q}}\otimes_\bb{Q} \bb{C},
\end{equation}
and exponential periods may be defined as the complex numbers needed to describe this map, that is to say, the entries of its matrix with respect to a $k$-basis of $H^i_{dR} \otimes_k \bb{C}$ and a $\bb{Q}$-basis of $H^i_{rd} \otimes_\bb{Q} \bb{C}$. \par

In \cite{KontSoib11} and subsequently in the book \cite{FJexp}, many tools to study these numbers are developed, including a category of exponential motives and various realisation functors emanating from this category. In particular, there is a \emph{perverse realisation} of an exponential motive; this is a perverse sheaf on the affine line $\bb{A}^1$ with vanishing cohomology. The rapid decay cohomology of $(X, f)$ can be recovered from it. Such sheaves possess pleasant properties, making their study a fruitful one, for example when calculating Galois groups of exponential motives. \par

If $f$ is not locally constant, the differential $\diff - \diff f$ defines a connection $\mcal{O}_Xe^f$ with irregular singularities at infinity. Far more generally however, any algebraic connection $\mcal{F}$ induces an associated de Rham complex, whose cohomology groups are finite-dimensional $k$-vector spaces. In \cite{Hien09}, Hien developed a dual rapid decay homology theory $H^{rd}_\bullet(X, \mcal{F})$, completing the picture begun by the above authors. There exists a canonical isomorphism
\begin{equation}
    H^i_{dR}(X, \mcal{F}) \otimes_k \bb{C} \xrightarrow{\sim} H^i_{rd}(X, \mcal{F}).
\end{equation}
In the case that $\mcal{F} = \mcal{O}_Xe^f$ is the connection inducing the above twisted de Rham complex, this rapid decay cohomology theory agrees with that described above after extending scalars to $\bb{C}$, i.e. $H^i_{rd}(X, \mcal{F}) = H^i_{rd}(X, f)_\bb{Q} \otimes_\bb{Q} \bb{C}$ and the induced period isomorphism agrees with \eqref{eqn: PP exp}. It is then natural to ask if $H^i_{rd}(X, \mcal{F})$ can be endowed with a $\bb{Q}$-structure, or perhaps with a $K$-structure where $K/\bb{Q}$ is a `small' field extension, in order to define a new class of periods. \par

In general, there is no obvious way to do this. The most natural approach would be to associate to the complex of sheaves which calculates rapid decay cohomology a $K$-structure, for which it is necessary, although \emph{not sufficient} (see \S \ref{Sec: lit review} below), to endow the local system of horizontal sections of $\mcal{F}$ with a $K$\nobreakdash-structure. In the classical case, this local system is $\bb{C}$, for which we can naturally choose the basis $1 \in \bb{C}$ to define our $\bb{Q}$-structure. In the exponential case, the local system is $\bb{C}e^f$, for which we can naturally choose the basis $1\cdot e^f$. For an arbitrary connection however, there are no forthcoming horizontal sections that make for a particularly `natural choice' of generators which could be used to define periods. This choice is vital, since for example in the classical case, had we chosen a number $z \in \bb C$ to define our $\bb Q$-structure as $\bb Q \cdot z \subset \bb C$, we would find $z$ to be a period of $\spec \bb Q$. Thus with an appropriate choice of $\bb Q$-structure, any complex number is a period. Of course, we wish to exclude these unnatural choices. \par 

An $E$-function is a power series $$F(z) = \sum_{n \geq 0} \frac{a_n}{n!}z^n$$ with algebraic coefficients satisfying
\begin{enumerate}
    \item $\Psi F = 0$ for some algebraic differential operator $\Psi \in \qbar[z, \partial_z]$,
    \item a growth condition on the $a_n$ and their denominators (see \cite{andré00} for the precise definition).
\end{enumerate}
$E$-functions generalise the exponential function; indeed, taking $a_n = 1$ gives $F(z) = e^z$. Other examples include the Bessel functions and the exponential integral function. Much is known about the transcendence properties of their special values, most notably the celebrated Siegel-Shidlovskii theorem \cite{NestShid96, andre00_2, Beukers06}. \par

$E$-operators, as defined by Andr\'e \cite{andré00}, are certain algebraic differential operators $\Psi \in \qbar[z, \partial_z]$ without singularities outside $\{0, \infty\}$ such that all solutions of the differential equation $\Psi F=0$ are given by $\bb C$-linear sums and products of $E$-functions, rational powers of $z$ and the logarithm $\log z$. All $E$-functions are annihilated by some $E$-operator \cite[Th\'eor\`eme 4.2]{andré00}. 

\subsection{Overview of this article}

In this paper, we consider a special case in which there is a natural way to choose a rational structure over a small field, namely some cyclotomic field $\bb Q(\zeta)$, where we take a connection $\mcal E$ on $\bb{G}_m$ of \emph{exponential type}, satisfying two additional mild conditions. The connection associated to an $E$-operator automatically satisfies two of the three conditions. Considered as meromorphic connections on $\bb P^1$, such objects have a regular singularity at $0$ and an irregular singularity at $\infty$. Under an assumption on $\mcal E$ (see \ref{Assump: half plane}), we endow the rapid decay cohomology groups of $f^*\mcal{E}$ with a rational structure over $\bb Q(\zeta)$, where $f: X \to \bb{G}_m$ is a non-vanishing regular function. For this, we use a technical result in \cite{CHH20}. This permits us to define a new class of periods. In the case that $\mcal{E}$ is associated to an $E$-operator, we call the complex numbers appearing in the resulting pairing \emph{$E$-periods} (see Definition \ref{Defn: E period}). \par

Nori's basic lemma is the most vital ingredient in the construction of the monoidal structure on the category of classical Nori motives \cite[\S 2.5]{HMS17}. It was generalised in \cite[\S 3.3]{FJexp} to the exponential case, where it serves a similar purpose in the category of exponential motives. In the process of defining our periods, we shall find that we may, analogously to the method in \cite{FJexp}, construct a perverse sheaf on $\bb{A}^1$ associated to the triple $(X, f, \mcal{E})$ from which rapid decay cohomology can again be recovered. As an almost immediate consequence of a famous result of Beilinson, we shall prove a basic lemma in our setting (Corollary \ref{Corr: basic lemma}). Furthermore, we find through our investigation that one can define rapid decay cohomology in this setting even for singular varieties (see Defn. \ref{Defn: sing RDC}). This may open the door to a motivic description of $E$-periods \`a la Nori. \par

This work suggests some natural questions: for example, how does the ring of $E$-periods relate to the ring of exponential periods? Exponential periods are known to be closely linked to the ring of \emph{$E$-values}, the complex numbers obtained by evaluating an $E$-function at an algebraic argument. In an upcoming work, Fres\'an and Jossen prove very concrete results in this direction, a couple of which are already stated in the notes \cite[Thm. 1.6]{Fresannotes}, see also Section \ref{Sec: Conc}. Whilst it is not thought that all exponential periods are $E$-values or vice versa, both form subsets of the $E$-periods and thus the latter could form a natural setting in which to study these relationships. We return to these questions in the conclusion. \par

\subsection{Comparison with previous work} \label{Sec: lit review}

The main inspiration for this article is the body of work on algebraic connections twisted by a regular function and the associated tools. In vague terms, what we show here is that some methods in the twisted case can be applied in a similar way to study a class of connections that are, on the face of it, significantly more complicated. The main obstruction is the Stokes phenomenon, which states that the behaviour of a solution to a differential equation around an irregular singularity is usually extremely erratic, which \emph{a priori} makes its rapid decay cohomology groups very difficult to calculate. \par

In \cite{FSY23}, a similar line of investigation is carried out. They also endow their cohomology groups with a rational structure over a field $K$ under a certain assumption (Assumption 2.22 in loc. cit.)\footnote{In the notation of Sec. \ref{sec: Conn exp type} in this paper, their assumption is equivalent to the given $K$-structure on $L$ descending to $L^{mod}$.}. In Hien's original article \cite[\S 3.3]{Hien09}, a similar assumption is made in order to define a `period determinant'. If the connection is irregular however, this field must usually be taken to be very large in order to allow the assumption to be satisfied. To see this, one does not need to look further than the main example in this paper: $E$-operators. In \cite{FischRiv14}, it is shown that the so-called \emph{Stokes constants} of an $E$-operator $\Psi$ lie in the ring
\begin{equation}
    \bb Q[\{G\text{-values}\}, \{\Gamma(a)_{a \in \bb Q\cap (0, 1)}\}, \gamma],
\end{equation}
where $G$-values are defined as the complex numbers obtained by evaluating a $G$-function at an algebraic argument and $\gamma$ is the Euler constant. Conjecturally, $G$-values are exactly the classical periods; needless to say, this ring contains many transcendental numbers. If we wish to have a $K$-structure compatible with a basis as in \eqref{eqn: andre basis 1}, $K$ must contain all of the Stokes constants of $\Psi$. To get around this, one would need to find a different basis such that the respective Stokes constants lie in a `small' field. To verify to what extent this is possible would be arduous, and the `optimal' choices would likely not be natural or satisfactory. \par

In this paper, we concentrate on a particular class of connections and replace the assumption in loc. cit. by another (\ref{Assump: half plane}), which is satisfied very often and is elementary to check. Under this assumption, \emph{any} $K$-structure on the sheaf of horizontal sections induces canonically a $K$-structure on the rapid decay cohomology groups (Thm. \ref{Thm: main result}). In other words, the Stokes constants do not come into play. This means we can take $K$ to be vastly smaller than typically necessary to fulfil the assumption in either of the aforementioned papers; in fact, the only limiting factor is the monodromy of the horizontal sections. If the monodromy representation has algebraic eigenvalues for example, it is possible to take $K$ to be a number field.

\subsection*{Outline of structure}

In Section \ref{sec: Conn exp type}, we perform the calculations that will allow us to endow the rapid decay cohomology groups with an $K$-rational structure. In Section \ref{Sec: review of perv}, we review the category $\text{Perv}_0$ and a construction of the perverse realisation of a variety with potential. In Section \ref{Sec: gen to other LS}, we show that applying a similar construction to a larger class of local systems $L$ on $\bb{G}_m$ (instead of the trivial one) also results in a perverse sheaf. We next use the findings of Section \ref{sec: Conn exp type} in Section \ref{Sec: comp to perv} to interpret certain stalks of the perverse sheaves we constructed in Section \ref{Sec: gen to other LS} as the rapid decay cohomology of the connection $f^*\mcal{E}$, where $f, \mcal{E}$ are as above. As an immediate corollary of this, we show that a basic lemma holds also in our setting. In Section \ref{Sec: rational structures} we present our main result and define $E$-periods. Finally, in Section \ref{Sec: Conc} we speculate about the relationship between exponential periods, the ring of $E$-values (special values of $E$-functions) and $E$-periods.

\subsection*{Acknowledgements}

I thank my PhD supervisor, Annette Huber, for suggesting that I investigate periods of irregular connections. I am grateful also to Javier Fres\'an for his guidance and many enlightening discussions. \par

\vspace{0.3cm}

In keeping with the guidelines of the Leiden Declaration \cite{Leiden}, the author hereby discloses that Artificial Intelligence, namely Claude (Anthropic), was used in the research leading to this manuscript solely as an occasional tool to survey literature, that is, for enquiries pertaining to the existence of results and research directions. All text, images, proof ideas and examples appearing in the following text were produced without the help of AI.

\subsubsection*{Notation and conventions}

Throughout, we stick to the following notation:
\begin{enumerate}
    \item $k$ will denote a number field with a fixed (implicit) embedding $k \hookrightarrow \bb C$,
    \item (unless otherwise stated) $X$ is a smooth variety over $k$ and $Y \subset X$ a divisor with simple normal crossings,
    \item $p_1: X \times \bb{A}^1 \to X$ and $p_2: X \times \bb{A}^1 \to \bb{A}^1$ denote the projections,
    \item (unless otherwise stated), $f$ will be a morphism $f: X \to \bb{G}_m$,
    \item $\Gamma \subset X \times \bb{A}^1$ denotes the graph of the composite $X \xrightarrow{f} \bb{G}_m \hookrightarrow \bb{A}^1$.
\end{enumerate}
Although we work almost exclusively with algebraic vector bundles with connection, we use the convenient language of $\D{}$-modules throughout. \par
Usually, we work in the analytic topology of each variety. However, all $\D{}$-modules are assumed to be the analytification of algebraic $\D{}$-modules defined over $k$. Furthermore, a constructible sheaf is taken to be constructible with respect to an \emph{algebraic} stratification. \par
Let $A \subset Z$ be a closed subset of a topological space $Z$. We denote the embeddings of $A$ and its complement as $A \xrightarrow{\iota_A} Z \xleftarrow{j_A} Z \setminus A$ and for a sheaf $\mscr{F}$ on $Z$ we use the notation
\begin{equation} \label{eqn: [Z, A]}
    \mscr{F}_{[Z, A]} := (j_A)_!(j_A)^{-1} \mscr{F}.
\end{equation}
Note that the functor $\mscr{F} \to \mscr{F}_{[Z, A]}$ is exact. For any sheaf of abelian groups $\mscr{F}$ on $Z$, there is the associated exact sequence
\begin{equation}
    0 \to \mscr{F}_{[Z, A]} \to \mscr{F} \to (\iota_A)_*\iota_A^{-1}\mscr{F} \to 0,
\end{equation}
which induces the long exact sequence in relative cohomology with coefficients in $\mscr{F}$. We shall use the fact that given another topological space $Y$ and a map $g: Y \to Z$, we have
\begin{equation} \label{eqn: [Z, A] pullback}
    g^{-1}\mscr{F}_{[Z, A]} \simeq (g^{-1}\mscr{F})_{[Y, g^{-1}(A)]}.
\end{equation}

In the definition of rapid decay cohomology, the \emph{oriented real blow-up} of a variety $\overline{X}$ along a simple normal crossings divisor $D \subset \overline{X}$ is required, which we write as $\text{OBl}_D(\overline{X})$. We set throughout
\begin{equation}
\rob := \text{OBl}_{\infty}(\bb{P}^1\setminus \{0\}),
\end{equation}
which is diffeomorphic to a closed punctured disc. In particular, it is \emph{not} $\text{OBl}_{\{0, \infty\}}\bb P^1$, which is a punctured disc compactified by two circles; one at its centre, one at its outer boundary. \\

\vspace{1pt}

\noindent \textbf{N.B.} In this paper, $\overline{X}$ is not typically compact; a compactification of a variety $X$ is written as $\overline{X}_c$. \par

\section{Connections of exponential type} \label{sec: Conn exp type}

In this section, we recall Hien's notion of rapid decay homology \cite{Hien09} and give the definition of the dual cohomology theory along with its relative version. We then take a connection of exponential type $\mcal{E}$ on $\bb{G}_m$ satisfying a certain condition (see \ref{Assump: half plane}) and construct a new sheaf that also calculates rapid decay cohomology of $f^*\mcal{E}$, where $f: X \to \bb{G}_m$ is a morphism. In practice, we are interested in the case that $\mcal E$ is associated to an $E$-operator, but here we can be more general at no extra cost.

\subsection{Rapid decay cohomology} \label{Subsec: RDC}

Let $\mcal F$ be an integrable connection on a smooth variety $X$ over a number field $k$, that is, a $\D{X}$-module which is locally free as an $\mcal{O}_X$-module. Hien \cite{Hien09} defines certain so-called \emph{rapid decay homology} groups $H^{rd}_i(X, \mcal F)$ which are constructed so that there is a canonical perfect pairing between them and the algebraic de Rham cohomology groups of $\mcal F$. To the knowledge of the author, the definition of the dual rapid decay \emph{cohomology} groups does not exist explicitly in the literature in generality, nor that of \emph{relative} rapid decay (co)homology. We thus begin with a brief recapitulation of Hien's theory and give the definitions of these groups. \par

Let $\overline{X}_c \supset X$ be a so-called \emph{good compactification} with respect to $\mcal F$ as defined by Mochizuki (see \cite[\S 3]{Hien09}) - such things always exist. The precise definition is rather involved, but the most important properties are that the complement $D := \overline{X}_c\setminus X$ is a divisor with simple normal crossings and that the formal completion along each component of $D$ decomposes into products of regular singular connections and connections of the form $\mcal{O}e^f$, just as in the one-dimensional Levelt-Turritin case (Eqn. \eqref{eqn: general formal decomp}). \par 

We set $(\widetilde{X}_c, \pi: \widetilde{X}_c \to \overline{X}_c)$ to be the oriented real blow-up along the normal-crossings divisor $D := \overline{X}_c\setminus X$ (see \cite[\S 4]{CHH20} or \cite[\S 3.4]{FJexp}). This is a real manifold with corners. There is a natural open embedding of the (analytified) variety $\tilde{\jmath}: X \hookrightarrow \widetilde{X}_c$ into the blow-up. \par

Let $\mcal{A}^{\text{mod}D}_{\widetilde{X}_c}$ (resp. $\mcal{A}^{<D}_{\widetilde{X}_c}$) be the sheaf of $C^\infty(\widetilde{X}_c)$ functions of moderate growth along $\widetilde{D} := \pi^{-1}(D)$ (resp. flat functions on $\widetilde{D}$), meaning growth not faster than polynomials in local coordinates (resp. meaning all derivatives tend to $0$ upon approaching the divisor), whose restrictions to $X \subset \widetilde{X}_c$ are holomorphic. To $\mcal F$ there is an associated analytic de Rham complex $\text{DR}(X, \mcal F)$ on $X$. We can also define its \emph{moderate} and \emph{asymptotically flat} de Rham complexes as 
\begin{equation}
\begin{aligned}
    \text{DR}^{\text{mod} D}(\widetilde{X}_c, \widetilde{D}, \mcal F) &:= \mcal{A}^{\text{mod}D}_{\widetilde{X}_c} \otimes_{\pi^{-1}\mcal{O}_{\overline X_c}} \pi^{-1}\text{DR}(\overline{X}_c, D, \mcal F) \\
    \text{DR}^{<D}(\widetilde{X}_c, \widetilde{D}, \mcal F) &:= \mcal{A}^{<D}_{\widetilde{X}_c} \otimes_{\pi^{-1}\mcal{O}_{\overline X_c}} \pi^{-1}\text{DR}(\overline{X}_c, D, \mcal F)
\end{aligned}
\end{equation}
respectively. Denote by $\mcal F^\vee := \mcal{H}om_{\mcal{O}_X}(\mcal F, \mcal{O}_X)$ the dual $\D{X}$-module. Hien \cite[Theorem 3]{Hien09} proves there is a pairing
\begin{equation}
    \text{DR}^{\text{mod} D}(\widetilde{X}_c, \widetilde{D}, \mcal F) \otimes_\bb{C} \text{DR}^{<D}(\widetilde{X}_c, \widetilde{D}, \mcal F^\vee) \to \tilde{\jmath}_!\bb{C}
\end{equation}
which is perfect is the sense that the induced maps
\begin{equation}
\begin{gathered}
    \text{DR}^{\text{mod} D}(\widetilde{X}_c, \widetilde{D}, \mcal F) \to R\mcal{H}om(\text{DR}^{<D}(\widetilde{X}_c, \widetilde{D}, \mcal F^\vee), \tilde{\jmath}_!\bb{C}),  \\
    \text{DR}^{<D}(\widetilde{X}_c, \widetilde{D}, \mcal F^\vee) \to R\mcal{H}om(\text{DR}^{\text{mod} D}(\widetilde{X}_c, \widetilde{D}, \mcal F), \tilde{\jmath}_!\bb{C})
\end{gathered}
\end{equation}
are isomorphisms in $D^b(\bb{C})$. \par
\emph{Rapid decay homology} is defined as
\begin{equation}
    H_i^{rd}(X, \mcal F):= R^{-i}\Gamma(\widetilde{X}_c, (\mcal{C}^{rd}_{\widetilde X_c})^\bullet(\mcal F))
\end{equation}
where $(\mcal{C}^{rd}_{\widetilde X_c})^\bullet(\mcal F)$ is a complex of sheaves which, roughly speaking, consists of tensor products of standard topological simplices with horizontal sections of rapid decay along $\widetilde{D}$. See \S 5.1 in loc. cit. for the precise definition. Since the resulting cohomology groups are independent of the compactification $(\overline{X}_c, D)$ (Corollary 1 in loc. cit.), we omit this from the notation. Furthermore, there is an isomorphism (Proposition 2 in loc. cit.)
\begin{equation}
    (\mcal{C}^{rd}_{\widetilde X_c})^\bullet(\mcal F) \simeq \text{DR}^{<D}(\widetilde{X}_c, \widetilde{D}, \mcal F)[2d].
\end{equation}
We therefore define \emph{rapid decay cohomology} as
\begin{equation}
    H^i_{rd}(X, \mcal F) := R^i\Gamma\left(\widetilde{X}_c, R\mcal{H}om((\mcal{C}^{rd}_{\widetilde X})^\bullet(\mcal F)[-2d], \tilde{\jmath}_!\bb{C})\right).
\end{equation}
Putting the above results together, we see that
\begin{equation}
    H^i_{rd}(X, \mcal F) \simeq R^i\Gamma\left(\widetilde{X}_c, \text{DR}^{\text{mod} D}(\widetilde{X}_c, \widetilde{D}, \mcal F)\right).
\end{equation}
Let $L := \tilde\jmath_*\mcal H^0\left(\text{DR}(X, \mcal F)\right)$ be the local system of horizontal sections of $\mcal F$ pushed forward to $\widetilde X_c$. Proposition 1 in loc. cit. states that the moderate de Rham complex has non-zero cohomology only in degree $0$. More precisely,
\begin{equation}
    \text{DR}^{\text{mod} D}(\widetilde{X}_c, \widetilde{D}, \mcal F) \simeq L^{\text{mod}},
\end{equation}
that is, $L^{mod} \subset L$, the sheaf of horizontal sections of $\mcal F$ with moderate growth along $\widetilde{D}$, calculates rapid decay cohomology. \par
Let $Y \subset X$ be a closed subvariety. We define $\widetilde{Y}_c \subset \widetilde{X}_c$ to be the closure of $Y$ inside $\widetilde{X}_c$ and we let $\tilde\iota: \widetilde{Y}_c \to \widetilde{X}_c$ be the inclusion.

\begin{definition}[Relative rapid decay cohomology] \label{Defn: relative RDC}
    Let $\mcal F$ be a connection on $X$ and let $\overline{X}_c \supset X$ be a good compactification with respect to $\mcal F$. We define the \emph{relative rapid decay cohomology} of the triple $(X, Y, \mcal F)$ as
    \begin{equation}
        H^i_{rd}(X, Y, \mcal F) := R^i\Gamma\left(\widetilde{X}_c,\, \text{Tot} \left(\text{DR}^{mod}(\widetilde{X}_c, \widetilde{D}, \mcal F) \to \tilde{\iota}_*\tilde{\iota}^*\text{DR}^{mod}(\widetilde{X}_c, \widetilde{D}, \mcal F) \right)\right),
    \end{equation}
    viewing the morphism of de Rham complexes as a double complex. Using the above results from \cite{Hien09}, we see, recalling the notation \eqref{eqn: [Z, A]},
    \begin{equation}
        H^i_{rd}(X, Y, \mcal F) \simeq R^i\Gamma\left(\widetilde{X}_c, L^{mod}_{[\widetilde X_c, \widetilde Y_c]}\right).
    \end{equation}
\end{definition}

\subsection{Connections of exponential type} \label{Subsec: conns of exp type}

Let $\mcal{F}$ be a meromorphic connection over $k(t)$. We may consider the formal completion over a ramified cover $\widehat{\mcal{F}}_b := k((t^{\frac{1}{b}})) \otimes_{k(t)} \mcal{F}$, $b \geq 1$. Recall the following well-known result of Levelt and Turrittin \cite{Turrittin1955, Lev75}, see also \cite{weatherhog}. For some (minimal) $b$ there exist regular singular connections $R_i$ over $\bb{C}(t^{\frac{1}{b}})$ and distinct exponents $f_i \in k((t^{\frac{1}{b}}))/t^{-1}k[[t^{\frac{1}{b}}]]$ such that
\begin{equation} \label{eqn: general formal decomp}
    \widehat{\mcal F}_b \simeq \bigoplus_{i\in I} R_i \otimes_{k(t^{\frac{1}{b}})} k(t^{\frac{1}{b}})e^{f_i},
\end{equation}
where $I$ is an index set and $k(t^{\frac{1}{b}})e^{f_i}$ is the $k(t^{\frac{1}{b}})\langle \partial \rangle$-module whereby the differential $\partial: k(t^{\frac{1}{b}}) \to k(t^{\frac{1}{b}})$ acts as $\partial(g) = \frac{\partial g}{\partial t} - g\frac{\partial f_i}{\partial t}$. \par \vspace{0.2cm}
We set $r_i$ to be the rank of $R_i$. Let $j: \bb{G}_m \to \bb{P}^1$ be the inclusion.
\begin{assumption}[\textbf{A}] \label{Assump on E}
We make the following set of assumptions about all $\D{\bb{G}_m}$-modules $\mcal E$ that we consider. We refer to them collectively as $\mathbf{A}$.
    \begin{enumerate}
    \item $j_+\mcal{E}$ has no singularities outside $\{0, \infty\}$ and the singularity at $0$ is regular.
    \item The singularity at $\infty$ is of \emph{exponential type}, meaning there is a formal decomposition as in \eqref{eqn: general formal decomp} of the form
    \begin{equation}
        (j_+\mcal E)_\infty \otimes k((z^{-1})) \simeq \bigoplus_{i\in I} R_i \otimes k(z^{-1})e^{\lambda_iz}
    \end{equation}
    for some $\lambda_i \in \qbar^\times$. In other words, $b=1$ and the exponents are linear in $z$.
\end{enumerate}
\end{assumption}
In Section \ref{Sec: comp to perv}, we make the following additional assumption:
\begin{enumerate}
    \item[($*$)] The associated local system of horizontal sections associated to $\mcal E$ has quasi-unipotent monodromy.
\end{enumerate}
The assumptions $\mathbf{A}$ together with $(*)$ are denoted by $\mathbf{A}^{qu}$.

\begin{remark}
    The assumptions $\mathbf{A}$ imply (see e.g. \cite[\S 18]{lodayrichaud91}) that there exists a formal basis of solutions at $\infty$ of the form
    \begin{equation} \label{eqn: basis at infty}
        \left(\mfr{f}_1\left(z^{-1}\right), \dots, \mfr{f}_n\left(z^{-1}\right)\right) e^{\Delta z} (z^{-1})^{\Gamma_\infty}
    \end{equation}
    where $\Delta$ is a diagonal matrix with diagonal entries $\lambda_i$, each $\lambda_i$ appearing with multiplicity $r_i$, $\Gamma_\infty$ is an upper-triangular matrix with entries in $\bb Q$ such that $\Delta$ and $\Gamma$ commute and
    \begin{equation}
        \mfr{f}_j\left(z^{-1}\right) \in \bb{C}\left[\left[z^{-1}\right]\right]
    \end{equation}
    are Gevrey series of order $1$. This means that the coefficients of the power series do not grow faster than $K^nn!$ for some $K > 0$.
\end{remark}

\begin{remark}[{\cite[Th\'eor\`eme 4.3]{andré00}}]
    The $\D{\bb G_m}$-module $\mcal E := \D{\bb G_m}/\D{\bb G_m}\Psi$ associated to an $E$-operator $\Psi$ satisfies $\qu$. In this case, we say $\mcal E$ is of \emph{type $E$}.
\end{remark}

We set 
\begin{equation}
    \rob := \text{OBl}_{\infty}(\bb{P}^1 \setminus \{0\}).
\end{equation}
Thus $\rob$ is diffeomorphic to a closed punctured disc. We use the following notation for sectors of opening $\geq \pi$ on the circle at infinity:
\begin{equation}
    S_{\theta, \varepsilon} := \left\{z \in \partial\rob \, \Big\vert \ \ \abs{\arg{z} - \theta} < \frac{\pi}{2} + \varepsilon\right\} \subset \rob,
\end{equation}
and set $S_\theta := S_{\theta, 0}$. We call such sectors \emph{large}. \par
In the following, we write (see Figure 1)
\begin{align*}
    \theta_i &= -\arg \lambda_i \in \bb{R}/2\pi i  \bb{Z} \\
    \rho_i &= e^{i\theta_i}\cdot \infty \in S^1 \cdot \infty \equiv \partial \rob \\
    S_i &= S_{\theta_i}
\end{align*}
This means $S_i$ is the sector of $\partial\rob$ on which $e^{\lambda_iz}$ \emph{grows} (not decays) exponentially and $\rho_i$, its centre, is the direction at infinity in which $e^{\lambda_i z}$ grows most quickly. We also abuse notation and tacitly consider $L$ as a sheaf on $\rob$ (meaning its direct image via $\bb{G}_m \hookrightarrow \rob$). \par

\usetikzlibrary{math, angles, calc}

\tikzmath{\x1 = 2.2; \y1 = 1.5; \t = atan(\y1/\x1); \R = 5; \r = 0.7;}

\begin{figure}[h]
	\centering
	\begin{tikzpicture}
        \filldraw[fill=yellow!10!white, draw=black] (5,5) circle (\R cm);
        \draw[dashed] (5, 5) -- ({(5+\R*cos(\t))}, {(5+\R*sin(\t))});
        \draw[dashed] (5, 5) -- (5+\x1, 5-\y1);
        \draw (5-\R, 5) -- (5+\R, 5);
        \node at (4.5+\R, 5.2) {\large $\bb R$};
        \node at (7.5, 6.3) {$\bar\lambda_i$};        
        \draw ({5+\r*cos(\t)},{5-\r*sin(\t)}) arc (-\t:\t:\r);
        \draw[color = RoyalBlue, very thick] ({5+\R*cos(\t-90)},{5+\R*sin(\t-90)}) arc (\t-90:\t+90:\R);
        \draw ({5+\R*cos(\t-90)},{5+\R*sin(\t-90)}) -- ({5+\R*cos(\t+90)},{5+\R*sin(\t+90)});
                \filldraw[fill=black, draw=black] (7.2,6.5) circle (0.05cm);
        \node at (7.5, 3.5) {$\lambda_i$};
        \filldraw[fill=black, draw=black] (7.2,3.5) circle (0.05cm);
        \node at ({5+(\R+0.3)*cos(\t)}, {5+(\R+0.3)*sin(\t)}) {$\rho_i$};
        \filldraw[fill=black, draw=black] ({5+\R*cos(\t)}, {5+\R*sin(\t)}) circle (0.05cm);
        \node at ({5+1.43*\r*cos(\t/2)}, {5-1.43*\r*sin(\t/2)}) {$-\theta_i$};
        \node at ({5+1.4*\r*cos(\t/2)}, {5+1.4*\r*sin(\t/2)}) {$\theta_i$};
        \node at (2, 6) {\LARGE $\bb G_m$};
        \node at ({-\R+5.6}, 8.5) {\Large $\partial\rob$};
        \node[color = RoyalBlue] at ({5+1.07*\R}, 6) {\Large $S_i$};
        \draw
        (5,5) coordinate (O)
        ({5+cos(\t)},{5+sin(\t)}) coordinate (A)
        ({5+cos(\t+90)},{5+sin(\t+90)}) coordinate (B)
        pic [draw, thick, angle radius = 0.3cm, angle eccentricity=2mm] {right angle = A--O--B};
        \filldraw[fill=white, draw=black] (5,5) circle (0.07cm);    
    \end{tikzpicture}
    \caption{A diagram of $\rob$.}
\end{figure}

We shall need the following well-known result. 

\begin{proposition}[e.g. {\cite[\S 18]{lodayrichaud91}}]
    Let $\mfr{f} \in \bb{C}[[t]]$ be a Gevrey series of order $1$ that is a formal solution to an algebraic differential equation. Then $\mfr{f}$ is $1$-summable in all but finitely-many directions. That is, for $\theta \in \mscr{S}$ where $\mscr{S} \subset \bb{R}/2\pi\bb{Z}$ is some cofinite subset (described below), there exists some $\varepsilon > 0$ and some open $U \supset S_{\theta, \varepsilon}$ such that there is a unique continuous function $f: U \to \bb{C}$ holomorphic on $U \cap \bb{G}_m$ such that its asymptotic expansion at any $\rho \in S_{\theta, \varepsilon}$ is equal to $\mfr{f}$. Furthermore, $f$ satisfies the same differential equation as $\mfr{f}$ on $U$.
\end{proposition}

\begin{remark}
    If $\mfr{f}$ does not converge, the analytic continuation of $f$ must have non-trivial monodromy. For example, the Laplace transform of $\frac{1}{1+u}$ on the sector centred on the ray $\arg z = \theta$, $\theta\neq \pi$, is given by
    \begin{equation}
        f_\theta(t) = \frac{1}{t} \int_0^{\infty e^{i\theta}} \frac{1}{1+u}e^{-\frac{u}{t}} \diff u.
    \end{equation}
    This has the asymptotic expansion $\sum_{n \geq 0} (-1)^n n! t^n$ along paths approaching $t=0$ inside $\{t \in \bb C \big|\,\text{Re}\, te^{-i\theta} > 0\}$. Furthermore, for small changes $\theta \to \theta'$, in the intersection of their domains of validity, $f_\theta(t) = f_{\theta'}(t)$. Thus varying $\theta$, it defines a function on $\bb{C} \setminus \bb{R}_{\leq 0}$. However, it picks up an extra term when crossing $\theta = \pi$:
    \begin{equation}
        f(e^{2\pi i}t) = f(t) + \frac{1}{t}e^{\frac{1}{t}}
    \end{equation}
    where $f(e^{2\pi i}t)$ denotes the analytic continuation of $f$ obtained by circling the origin once in the anticlockwise direction. The extra contribution comes from a contour integral around $u = -1$.
\end{remark}

\begin{remark}[Stokes phenomenon] \label{Rem: Stokes phenom}
    For $\mcal{E}$ as above, the directions at infinity in which $\mfr{f}_i$ are not 1-summable are called \emph{anti-Stokes lines}, and it is at these lines that the asymptotic behaviour of a germ of a horizontal section of $\mcal{E}$ may change. More precisely, consider a horizontal section of the form $f(z)e^{\lambda_iz}$, $f$ analytic and of moderate growth on some sector centred on a point $\sigma_1 \in \partial \rob$ just to the side of an anti-Stokes line $\bb{R}_{>0}e^{i\theta}$, $\theta \in \bb{R}/2\pi i\bb{Z}$. We may consider its analytic continuation to a point $\sigma_2$ the other side of the anti-Stokes line. Then its asymptotic expansion on a large sector centred on $\sigma_2$ may differ from that at $\sigma_1$, i.e. it may be of the form  
    \begin{equation}
        g(z)e^{\lambda_iz} + \sum_{j \neq i}h_j(z)e^{\lambda_jz} = e^{\lambda_iz}\left(g(z) + \sum_{j \neq i} h_j(z)e^{(\lambda_j-\lambda_i)z}\right)
    \end{equation}
    for some $g, h_j$ analytic and of moderate growth on the second sector. However, on the overlap, these expansions must agree since they define the same function. By bringing $\sigma_1$ and $\sigma_2$ arbitrarily close to the anti-Stokes line, we can make the opening of the overlap of these sectors arbitrarily close to $\pi$. Then the equality
    \begin{equation}
        f(z) - g(z) = \sum_{j \neq i} h_j(z)e^{(\lambda_j - \lambda_i)z}
    \end{equation}
    holds on these overlaps. Since $f, g, h_j$ are of moderate growth on an open sector of opening $\pi$ centred on the anti-Stokes line, $e^{(\lambda_j - \lambda_i)z}$ must be as well for each $j$ such that $h_j \neq 0$. But this means that $(\lambda_j - \lambda_i)e^{i\theta} \in \bb{R}_{<0} \iff \theta = \pi - \arg{(\lambda_j - \lambda_i)} = \arg{(\bar{\lambda}_i - \bar{\lambda}_j)}$. Therefore, the set
    \begin{equation}
        \Sigma := \{\arg(\bar{\lambda}_i - \bar{\lambda}_j) \vert i, j \in I, i \neq j\} \subset \bb{R}/2\pi \bb{Z}
    \end{equation}
    contains all anti-Stokes directions of $\mcal{E}$. In words: when crossing the line $\theta = \arg(\bar{\lambda}_i - \bar{\lambda}_j)$, a term $\sim e^{\lambda_i z}$ may `pick up' a term $\sim e^{\lambda_j z}$.
\end{remark}

\begin{proposition}[Stokes decomposition, {\cite[\S 12]{malgrange91}}]
    Let $\theta \in (\bb{R}/2\pi i\bb{Z}) \setminus\Sigma$ be a non-anti-Stokes direction and set $\sigma = e^{i \theta} \cdot \infty$. Then there exists a unique direct sum decomposition of the stalk $L_\sigma$ of horizontal sections at $\sigma$ into vector spaces $\text{St}^i_\sigma$, $i \in I$, given by
    \begin{equation}
        L_\sigma = \oplus_{i \in I} \text{St}^i_\sigma
    \end{equation}
    such that each element of $\text{St}^i_\sigma$ can be written as $g(z)e^{\lambda_iz}$ on a large sector centred on $\sigma$, where $g$ is analytic and of moderate growth on the sector. Furthermore, $\dim \text{St}^i_\sigma = r_i = \text{\emph{rk}}\, R_i$ (see \eqref{eqn: general formal decomp}). This is the \emph{Stokes decomposition}.
\end{proposition}

\subsection{A simpler way to compute RD cohomology}

We now wish to investigate the rapid decay cohomology of the pullback connection $f^*\mcal{E}$, where $f: X \to \bb G_m$ is a regular, non-vanishing function and $\mcal{E} \in \modd{}{\D{\bb G_m}}$ satisfies $\mbf{A}$ as above. We construct in several steps a more tangible sheaf that also computes rapid decay cohomology of the triple $(X, Y, f^*\mcal{E})$, where $Y \subset X$ is a closed subvariety. Put simply, we reduce rapid decay cohomology in this setting to relative cohomology with coefficients in a local system. This culminates in Theorem \ref{Thm: L* calcs RDC}; the key result that gives us rational structures (Thm. \ref{Thm: main result}) and is also applied in Section \ref{Sec: comp to perv} to eventually prove the basic lemma. \par

In this section, we fix an $\mcal{E} \in \modd{}{\D{\bb{G}_m}}$ satisfying $\mbf{A}$. Following the notation of \S\ref{Subsec: RDC}, we set $L$ to be the local system of horizontal sections of $\mcal{E}$ on $\bb G_m$, or by abuse of notation its direct image under inclusion into either $\rob$ or $\text{OBl}_{\{0, \infty\}}\bb P^1$. \par

\subsubsection{Preliminaries} \label{Sec: prelim}

The expert and the uninterested reader may look at only \eqref{eqn: Xtilde defn}, the statement of Lemma \ref{Lem: RDC from pullback} and skip to Section \ref{Sec: results}. \par

First, we note that as in the twisted case \cite{CHH20, FJexp}, we do not have to use the `full' blow-up; we can safely discard parts of it without disturbing the cohomology groups of our sheaves. Take a compactification $\overline{X}_c \supset X$ with $D:=\overline{X}_c \setminus X$ a divisor with simple normal crossings such that $f$ extends to a map {$\bar{f}_c: \overline{X}_c \to \bb{P}^1$}. Let $(\widetilde{X}_c,\, \pi: \widetilde{X}_c \to \overline{X}_c)$ be the oriented real blow-up of $\widetilde{X}_c$ along $D$. There is a natural map $\tilde f_c: \widetilde{X}_c \to \text{OBl}_{\{0, \infty \}}\bb P^1$ lifting $\bar f_c$ \cite[Lemma 4.4]{CHH20}. We decompose $D = P \cup H$ where $P = \bar f_c^{-1}(\infty)$ and $H$ are both divisors with simple normal crossings. \par

The horizontal sections of $f^*\mcal E$ are given by $f^{-1}L$. Considering them as holomorphic functions, a horizontal section $h$ of $\mcal E$ gets pulled back to $h \circ f$. In particular, $h \circ f$ has moderate growth on an open set $V \subset \widetilde{X}_c$ if and only if $h$ has moderate growth on some open set containing $f(V) \subset \text{OBl}_{\{0, \infty\}}\bb P^1$. We can thus identify
\begin{equation}
    H^i_{rd}(X, f^*\mcal E) = H^i(\widetilde{X}_c, \tilde f_c^{-1}L^{mod}),
\end{equation}
where $L^{mod} \subset L$ is the subsheaf of horizontal sections of moderate growth on the boundary.

\vspace{10pt}
\begin{aside}
    If $\mcal E$ were a connection on $\bb A^1$, for example $\mcal O_{\bb A^1}e^z$, and $f$ were a morphism landing in $\bb A^1$, an equation of this form would still hold. This shows the equivalence between Hien's definition of rapid decay (co)homology in the twisted setting and the elementary definition via relative (co)homology \cite[Prop. 3.5.2]{FJexp}.
\end{aside}
\vspace{10pt}

Using a similar approach as in \cite[Defn. 6.3]{CHH20}, we set
\begin{equation} \label{eqn: Xtilde defn}
    \widetilde{X} := \widetilde{X}_c \setminus \pi^{-1}(H).
\end{equation}
By the Collar Neighbourhood Theorem \cite[Thm. 3.4.6]{FJexp}, the immersion $\tilde{\jmath}: \widetilde{X} \to \widetilde{X}_c$ is a homotopy equivalence. The regularity of $\mcal{E}$ at $0\in \bb{P}^1$ means precisely that $L^{mod}\vert_U = L\vert_U$ on a neighbourhood $U$ of the circle at $0$ in $\text{OBl}_{\{0,\infty\}}\bb{P}^1$, in particular, $L^{mod}\vert_U$ is a local system. Putting these facts together, and noting that $\bar f^{-1}(0) \subset H$, one can verify that the cohomology groups of $\tilde f_c^{-1}L^{mod}$ and $\tilde f^{-1}_cL^{mod}\vert_{\widetilde{X}}$ coincide:
\begin{equation}
    H^i(\widetilde{X}, \tilde f^{-1}_cL^{mod}\vert_{\widetilde{X}}) = H^i(\widetilde{X}_c, \tilde f^{-1}L^{mod}).
\end{equation}
The circle at $0$ does not lie in the image of $\widetilde{X}$ under $\tilde f_c$, thus we may set
\begin{equation} \label{eqn: tilde f}
    \tilde{f} := \tilde f_c\vert_{\widetilde{X}}: \widetilde{X} \to \rob.
\end{equation}
For a closed subvariety $Y \subset X$, we set $\widetilde Y$ (resp. $\widetilde{Y}_c$) to be its closure inside $\widetilde{X}$ (resp. $\widetilde{X}_c$). Then the inclusion $\widetilde Y \to \widetilde{Y}_c$ is also a homotopy equivalence. We summarise this discussion in the following.

\begin{lemma} \label{Lem: RDC from pullback}
Let $X$, $Y$ and $\mcal E$ be as above and we keep the notation of this subsection. Then $\tilde f^{-1}L^{mod}$ calculates rapid decay cohomology, that is,
\begin{equation}
    H^i_{rd}(X, Y, f^*\mcal E) \simeq H^i(\widetilde{X}, \widetilde{Y}, \tilde f^{-1}L^{mod}).
\end{equation}    
\end{lemma}

\subsubsection{Results} \label{Sec: results}

We now begin the work that provides the foundations for the main results of this paper. From now on, $L^{mod}$ is considered as a sheaf on $\rob$. 

\begin{lemma}\label{Lem: L^> isom}
There is an isomorphism
\begin{equation}
    \text{\emph{coker}} (L^{mod} \hookrightarrow L) \simeq \bigoplus_i (\iota_{\overline{S}_i})_*\bb{C}^{r_i}
\end{equation}
where $\iota_{\overline{S}_i}: \overline{S}_i \hookrightarrow \rob$ is the inclusion.
\end{lemma}
\begin{proof}
At each non-anti Stokes direction $\sigma \in \partial\rob$, we have the Stokes decomposition
\begin{equation} \label{eqn: Stokes decomp}
    L_\sigma = \oplus_{i \in I} \text{St}^i_\sigma
\end{equation}
where $\text{St}^i_\sigma$ consists of sections of $L$ whose analytic continuations to the sector $S_{\arg \sigma, \varepsilon}$ for some $\varepsilon > 0$ are of asymptotic growth $\sim e^{\lambda_i z}$. If $\iota: Y \to X$ is a closed immersion of topological spaces, $\mscr{F}$ a local system of $\bb{C}$-vector spaces on $X$, then a $\bb C$-linear morphism $\mscr{F} \to \iota_*\bb{C}$ is the datum of a $\bb{C}$-linear map $\Gamma(Y, \mscr{F}) \to \bb{C}$. We identify by parallel transport $L_{\rho_i} \equiv \Gamma(S_i, L)$. Assuming $\rho_i$ is not an anti-Stokes direction, we then define a morphism
\begin{equation}
    p_i: L \to (\iota_{\overline{S}_i})_*\bb{C}^{r_i}
\end{equation}
by sending each summand $\text{St}^j_{\rho_i}$, $j \neq i$ to zero and mapping $\text{St}^i_{\rho_i}$ isomorphically onto $\bb{C}^{r_i}$. If $\theta_i$ happens to be an anti-Stokes direction, choose a point $\rho_i'$ on $\partial\rob$ very close to $\rho_i$ and apply the same procedure with the Stokes decomposition at $\rho_i'$. This is dependent on which side of $\rho_i$ the chosen point $\rho_i'$ lies, but in an inconsequential way: by the discussion in Remark \ref{Rem: Stokes phenom}, carrying some $\varphi \in \text{St}^i_{\rho_i'}$ across the Stokes line gives something of the form
\begin{equation}
    f(z)e^{\lambda_iz} + \sum_l g_l(z) e^{\lambda_i(1-\alpha_l)z} = e^{\lambda_iz}\left(f(z) + \sum_l g_l(z)e^{-\alpha_l\lambda_iz}\right)
\end{equation}
with $\alpha_l > 0$ and $f, g_l$ analytic functions of moderate growth on $S_i$. Thus on the whole sector $S_i$, the function has asymptotic growth $\sim e^{\lambda_iz}$, which means the function has moderate growth nowhere on the sector. We have therefore constructed a map
\begin{equation} \label{eqn: surjection from L}
    P := \bigoplus_i p_i: L \to \bigoplus_{i \in I}(\iota_{\overline{S}_i})_*\bb{C}^{r_i}
\end{equation}
The subsheaf $L^{mod}$ lies in the kernel of $P$ since if a stalk of $L^{mod}$ had non-zero projection to some $\text{St}^i_{\rho_i}$ at a point $\sigma \in S_i$, it would have growth at least $\sim e^{\rho_i z}$, which is not moderate at any point in $S_i$. It remains to prove that $P$ is surjective. Take some $\sigma \in \partial \rob$ and some germ
\begin{equation}
    s \in ((i_{\overline {S_i}})_*\bb C^{r_i})_\sigma \subset \left(\bigoplus_{i \in I}(\iota_{\overline{S}_i})_*\bb{C}^{r_i}\right)_\sigma.
\end{equation}
We may consider $s$ as a germ at $\rho_i$ via parallel transport along a path in $\overline{S}_i$. Since $p_i\vert_{\rho_i}: L_{\rho_i} \to \bb C^{r_i}$ is a projection onto a direct summand, it admits a canonical section. We use this to lift $s$ to a function $\varphi \in \text{St}^i_{\rho_i} \subset L_{\rho_i}$, which we then parallel transport back to $\sigma$ along a path in $\overline{S}_i$. \par
We claim that $P\vert_\sigma(\varphi) = s$. To prove this, we must check that for any $*$ with $\overline{S_*} \ni \sigma$, the parallel transport of $\varphi$ to $\rho_*$ along a path in $\overline{S_*}$ has zero projection onto $\text{St}^*_{\rho_*}$. This can be checked by using Remark \ref{Rem: Stokes phenom} and considering the growth of terms that might be `picked up' as we analytically continue from $\rho_i$ to $\rho_*$: starting with $\varphi$ at $\rho_i$, the only terms that can be picked up due to the part $\sim e^{\lambda_iz}$ have growth $\sim e^{\lambda'z}$ where $\sim \bar\lambda'$ lies in the region below shaded green; meanwhile, the only terms that could pick up a term of growth $\sim e^{\lambda_*z}$ have growth $\sim e^{\lambda''z}$ where $\bar\lambda''$ lies in the region shaded pink. Since it is clear that any terms that are indeed picked up are likewise incapable of picking up a term of growth $\sim e^{\lambda_*}$, the claim is proved.

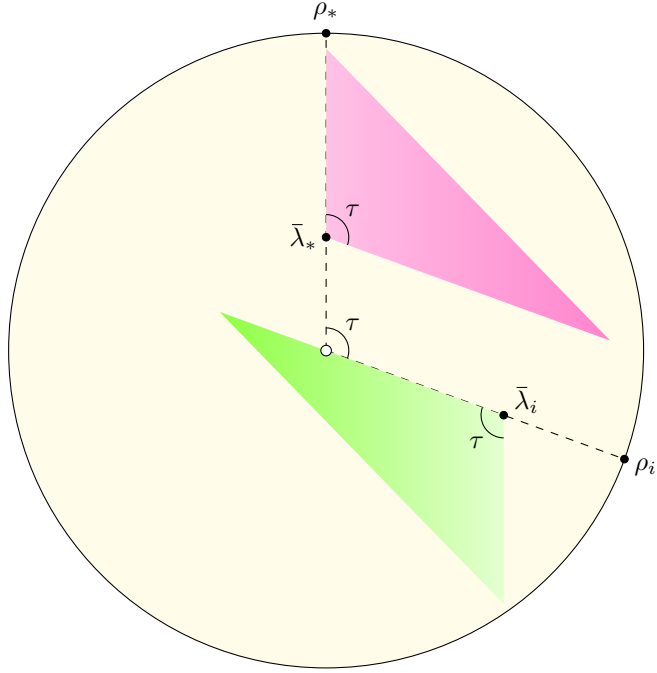
\tikzmath{\R = 4.2; \t = 110; \y1 = 6.5; \x2 = 2.5*cos(90-\t)+5; \y2 = 5+2.5*sin(90-\t); \r = 0.3;}

\begin{figure}[h]
	\centering
	\begin{tikzpicture}
        \filldraw[fill=yellow!10!white, draw=black] (5,5) circle (\R cm);
        \draw[dashed] (5, 5) -- ({\R*cos(90-\t)+5}, {\R*sin(90-\t)+5});
        \draw[dashed] (5, 5) -- (5, 5+\R);
        \node at (4.7, \y1) {$\bar \lambda_*$};
        \node at (\x2+0.3, \y2+0.2) {$\bar\lambda_i$};        
        \shade[left color=WildStrawberry!25!white,right color=WildStrawberry!50!white] (5,\y1+2.5) -- (5,\y1) -- ({5+4*cos(\t-90)},{\y1-4*sin(\t-90)});
        \shade[left color=YellowGreen,right color=YellowGreen!25!white] (\x2,\y2-2.5) -- (\x2,\y2) -- ({\x2+4*cos(-90-\t)},{\y2+4*sin(-90-\t)});
        \draw ({5+\r*cos(\t-90)},{\y1-\r*sin(\t-90)}) arc (90-\t:90:\r);
        \draw (\x2,{\y2-\r}) arc (-90:-90-\t:\r);
        \filldraw[fill=black, draw=black] (5, 5+\R) circle (0.05cm); 
        \filldraw[fill=black, draw=black] ({5+\R*cos(\t-90)}, {5+\R*sin(90-\t)}) circle (0.05cm); 
        \draw ({5+\r*cos(90-\t)},{5-\r*sin(\t-90)}) arc (90-\t:90:\r);
        \node at ({5+0.5*cos(45)}, {\y1+0.5*sin(45)}) {$\tau$};
        \node at ({\x2-0.5*cos(45)}, {\y2-0.5*sin(45)}) {$\tau$};
        \node at ({5+0.5*cos(45)}, {5+0.5*sin(45)}) {$\tau$};
        \node at ({5+(\R+0.3)*cos(\t-90)}, {5-(\R+0.3)*sin(\t-90)}) {$\rho_i$};
        \node at ({5}, {5+(\R+0.3)}) {$\rho_*$};
        \filldraw[fill=black, draw=black] (5, \y1) circle (0.05cm);
        \filldraw[fill=black, draw=black] (\x2, \y2) circle (0.05cm);
        \filldraw[fill=white, draw=black] (5,5) circle (0.07cm);    
    \end{tikzpicture}
\end{figure}

\end{proof}

We now make the following additional assumption:

\begin{assumption} \label{Assump: half plane}
    All $\lambda_j$ lie in a half-plane with boundary intersecting the origin, i.e. there exists an $\alpha \in \bb R/2\pi \bb Z$ with
    \begin{equation}
        \text{Re}(e^{i\alpha} \lambda_j) > 0 \text{ for all } j.
    \end{equation}
    Together with Assumption $\mbf{A}$ (resp. $\qu$), we denote this by $\mbf{A}_+$ (resp. $\qu_+)$.
\end{assumption}

Assumption \ref{Assump: half plane} is equivalent to $\cap_{j\in I} S_j \neq \emptyset$. The intuitive meaning of this is that there is some point in $\partial \rob$, namely $e^{i\alpha} \cdot \infty$, at which no horizontal section of $\mcal E$ has moderate growth whereas at the opposite point, $-e^{i\alpha}\cdot \infty$, every horizontal section has moderate growth.

\begin{theorem} \label{Thm: L* calcs RDC} 
    Assume $\mcal E \in \modd{}{\D{\bb G_m}}$ satisfies $\mbf A_+$ and let {$\sigma \in \cap_{j \in I} S_j$}. Let $X$ be smooth, $f: X \to \bb G_m$ a morphism and $Y \subset X$ a closed subvariety. We keep the notation from Lemma \ref{Lem: RDC from pullback}. Then there is a canonical isomorphism
    \begin{equation} \label{eqn: L* calcs RDC}
        H^i_{rd}(X, Y, f^*\mcal{E}) \simeq H^i\left(\widetilde{X}, \widetilde{Y} \cup \tilde f^{-1}(\sigma),\tilde f^{-1}L\right),
    \end{equation}
     for all $i \geq 0$.
\end{theorem}

\begin{proof}
    From Lemma \ref{Lem: L^> isom}, we have an exact sequence
    \begin{equation}
        0 \to L^{mod} \to L \to \bigoplus_j (i_{\overline{S}_j})_*\bb{C}^{r_j} \to 0.
    \end{equation}
    We apply the exact functor $(\cdot)_{[\rob, \sigma]}$. Noting that the stalk $L^{mod}_{\sigma} = 0$, it is clear this operation leaves $L^{mod}$ unchanged, resulting in
    \begin{equation}
        0 \to L^{mod} \to L_{[\rob, \sigma]} \to \bigoplus_j \bb{C}^{r_j}_{[\overline{S}_j, \sigma]} \to 0.
    \end{equation}
    We apply the exact functors $\widetilde{f}^{-1}$ and $(\cdot)_{[\widetilde{X}, \widetilde{Y}]}$ to get yet another exact sequence
    \begin{equation}
        0 \to \tilde f^{-1}L^{mod}_{[\widetilde{X}, \widetilde{Y}]} \to \tilde f^{-1}L_{[\widetilde{X}, \widetilde Y \cup \tilde f^{-1}(\sigma)]} \to \bigoplus_j \bb{C}_{[\tilde f^{-1}(\overline{S}_j), \widetilde Y \cup \tilde f^{-1}(\sigma)]}^{r_j} \to 0,
    \end{equation}
    where we have used Eqn. \ref{eqn: [Z, A] pullback} several times. As the cohomology of the middle term is precisely the RHS of \eqref{eqn: L* calcs RDC}, the theorem will follow if we show that the cohomology of the last term vanishes. \par
    First we assume $Y = \emptyset$. Note for example that if $\tilde f: \rob \to \rob$ is the identity, the cohomology of the last term vanishes since $\overline{S}_i$ is contractible. Proposition 10.4 in \cite{CHH20} essentially states that this lifts to $\widetilde{X}$. We require that the inclusion $\tilde f^{-1}(\sigma) \subset \tilde f^{-1}(\overline S_i)$ induces an isomorphism on cohomology. This is implied by the cited proposition together with the fact that, in the notation of loc. cit., the inclusions
    \begin{equation}
        B^\circ_{\overline{X}}(X, f) \hookrightarrow B_{\overline X}(X), \qquad \quad B^\#_{\overline{X}}(X, f) \hookrightarrow B_{\overline X}(X)
    \end{equation}
    are homotopy equivalences. This proves the result for $Y = \emptyset$. \par 
    Now assume $Y \subset X$ is non-empty. Using the above, it suffices by the long exact sequence in relative cohomology to prove that the inclusion $\tilde f^{-1}(\sigma) \cap \widetilde{Y} \subset \tilde f^{-1}(\overline S_i) \cap \widetilde{Y}$ induces an isomorphism on cohomology. Corollary 10.8 in loc. cit. gives the result we require as long as $Y$ is a divisor with simple normal crossings. Thus we must reduce to this case. \par
    We achieve this via a resolution of singularities and an excision argument. Below, we construct a desingularisation $(X_B, Y_B, f_B)$ with a proper map $\tau: X_B \to X$ such that $Y_B = \tau^{-1}(Y)$ has simple normal crossings and such that there is commutative diagram \\
    \adjustbox{center}{
    \begin{tikzcd}
        \widetilde{X}_B \arrow{rr}{\tilde\tau} \arrow{dr}{\tilde f_B} & & \widetilde{X} \arrow{dl}{\tilde f} \\
        & \rob &
    \end{tikzcd}
    }
    where $\tilde\tau$ is proper and $\widetilde{X}_B$ and $\tilde f_B$ have their obvious meanings. It suffices to demonstrate that 
    \begin{equation} \label{eqn: isom2}
        H^i(\widetilde X, \widetilde Y, \mscr F) \simeq H^i(\widetilde X_B, \widetilde Y_B, \tilde\tau^{-1}\mscr F),
    \end{equation}
    for any $i$ and any sheaf $\mscr F$ on $\widetilde{X}$. This is because we can then choose $\mscr F = \tilde f^{-1}L^{mod}$ or $(\tilde f^{-1}L)_{[\widetilde{X}, \tilde f^{-1}(\sigma)]}$ for the left- and right-hand sides of \eqref{eqn: L* calcs RDC} respectively. This follows by excision; see the proof of Proposition 2.18 in \cite{HMS17}. \par
    It remains only to construct $\tilde\tau$. Recall that we denote by $(\overline X_c, \bar f_c: \overline X_c \to \bb P^1)$ a good compactification of $(X, f)$. By resolution of singularities, we find a variety $\overline X_{B, c}$ equipped with a proper morphism $\tau_c: \overline X_{B, c} \to \overline X_c$ such that $\overline Y_B := \tau_c^{-1}(\overline Y)$ is a divisor with normal crossings and $\tau_c\vert_{\overline X_{B, c}\setminus \overline Y_B}: {\overline X_{B, c}\setminus \overline Y_B} \to \overline X_c \setminus \overline Y$ is an isomorphism. Let $\bar f_{B, c} = \bar f_c \circ \tau_c$. Setting $X_B = \tau_c^{-1}(X)$ and $Y_B = \tau_c^{-1}(Y)$, we can decompose the normal crossings divisor $D_B := \overline{X}_{B, c} \setminus X_B$ as $P_B \cup H_B$, with $P_B$ and $H_B$ both normal crossings divisors and $P_B = \bar f_{B,c}^{-1}(\infty)$. Note that $H_B = \tau_c^{-1}(H)$. \par
    By \cite[Lemma 4.4]{CHH20}, the map $\tau_c$ induces a (proper) map $\tilde \tau_c: \widetilde{X}_{B, c} \to \widetilde{X}_c$, satisfying $\tilde \tau_c^{-1}(\pi^{-1}(H)) = \pi_B^{-1}(H_B)$. In particular, recalling that $\widetilde{X} = \widetilde{X}_c \setminus \pi^{-1}(H)$ and $\widetilde{X}_B = \widetilde{X}_{B, c} \setminus \pi_B^{-1}(H_B)$, we find via restriction the required map $\tilde\tau$.
\end{proof}

Thm. \ref{Thm: L* calcs RDC} suggests the following definition of rapid decay cohomology for singular varieties in our setting. It is analogous to the elementary definition of Fres\'an and Jossen and has the advantage of manifest independence of a choice of compactification.

\begin{definition}[RDC for arbitrary varieties] \label{Defn: sing RDC}
    Let $S_r = \{z \in \bb C^\times \vert \text{Re}\, z \geq r\}$. Let $X$ be a variety over $k$, $Y \subset X$ a closed subvariety, $f: X \to \bb G_m$ a morphism and suppose $\mcal E$ satisfies \ref{Assump: half plane} for some $\alpha$. We define the rapid decay cohomology of the quadruple $(X, Y, f, \mcal E)$ as
    \begin{equation}
        H^i_{rd}(X, Y, f^*\mcal E) := \underset{r \to \infty}{\text{colim}} \, H^i(X, Y \cup f^{-1}(e^{i\alpha}S_r), f^{-1}L).
    \end{equation}
\end{definition}
\noindent Here, if $X$ is singular, $f^*\mcal E$ is to be understood in quotation marks. These groups do not depend on the choice of (valid) $\alpha$ and agree with the previous definition, since for an alternative choice $\alpha'$ and large enough $r$, the inclusions
\begin{equation}
    f^{-1}(s) \hookrightarrow f^{-1}(e^{i\alpha}S_r \cap e^{i\alpha'}S_r) \hookrightarrow f^{-1}(e^{i\alpha}S_r)
\end{equation}
are homotopy equivalences for any $s \in e^{i\alpha}S_r \cap e^{i\alpha'}S_r$. This follows from Ehresmann's fibration theorem \cite[\S 3.1.2]{FJexp}. \par

This definition is the one we use for the rest of this paper. Of course, if $X$ is nonsingular, it agrees with the previous definition by Thm. \ref{Thm: L* calcs RDC}.

\begin{remark} \label{Rem: dim of RD}
    Thm. \ref{Thm: L* calcs RDC} shows that the dimension of the rapid decay cohomology groups $H^i_{rd}(X, Y, f^*\mcal{E})$ depends only on the local system $L$.
\end{remark}

\begin{remark}
    One can obtain the same result as in Thm. \ref{Thm: L* calcs RDC} even assuming only $\mbf{A}$, not $\mbf{A}_+$. However, the isomorphism is non-canonical and the proof is far messier. Nonetheless, it suffices to prove Remark \ref{Rem: dim of RD} and also to prove the basic lemma in just the same way as is detailed below.
\end{remark}

For $z \in \bb G_m$ in a small neighbourhood of $\sigma \in \rob$, we have
\begin{equation} \label{eqn: isom}
    H^i\left(\widetilde{X}, \widetilde{Y} \cup \tilde f^{-1}(\sigma), \tilde{f}^{-1}L\right) \simeq H^i\left(\widetilde{X}, \widetilde{Y} \cup f^{-1}(z), \tilde{f}^{-1}L\right) \simeq H^i\left(X, Y \cup f^{-1}(z), f^{-1}L\right),
\end{equation}
where we use in the last isomorphism that $X \hookrightarrow \widetilde{X}$ and $Y \hookrightarrow \widetilde Y$ are homotopy equivalences. This gives us the following. 

\begin{corollary} \label{Corr: RDC isom FF}
    Continuing with the above notation, suppose $\mcal E$ satisfies $\mbf A_+$ for some $\alpha \in \bb R/2\pi \bb Z$. Let $z \in \bb C$ be such that $\text{Re}\, e^{i\alpha}z \gg 0$. Then we have a canonical isomorphism
    \begin{equation} \label{eqn: RDC isom FF}
        H_{rd}^i(X, Y, f^*\mcal E) \simeq R^i(p_2)_*((f\circ p_1)^{-1}L)_{[X \times \bb A^1, \Gamma\cup Y \times \bb A^1]}\big\vert_z.
    \end{equation}
\end{corollary}

Note that the right hand side of this equation is (canonically equivalent to) a nearby fibre at infinity of the perverse sheaf $$R^i(p_2)_*((f \circ p_1)^{-1}L)_{[X \times \bb{A}^1, \Gamma \cup Y \times \bb{A}^1]}[1],$$ discussed in Section \ref{Sec: gen to other LS} (Lemma \ref{Lem: GC vanishes 2}).

\section{Review of perverse sheaves on the affine line} \label{Sec: review of perv}

We briefly review some key results on perverse sheaves on the affine line. \par

Let $F^\bullet \in D^b_c(\bb{Q})$ be a complex of $\bb{Q}$-sheaves on $\bb{A}^1$ such that the cohomology sheaves are constructible. Then $F^\bullet$ is a perverse sheaf if and only if the following conditions hold:
\begin{itemize}
    \item $\mcal{H}^i(F^\bullet) = 0$ unless $i \in \{-1, 0\}$.
    \item $\mcal{H}^{-1}(F^\bullet)$ has no non-zero global sections with finite support.
    \item $\mcal{H}^0(F^\bullet)$ is a skyscraper sheaf.
\end{itemize}
The full subcategory $\text{Perv}_0 \subset \text{Perv}$ consisting of the objects with zero cohomology admits a very tidy description:
\begin{proposition}[{\cite[Thm. 2.29]{KKP08}}] \label{Prop: perv0 charac}
    Every $F^\bullet \in \text{Perv}_0$ has cohomology concentrated in degree $-1$. In other words, the objects of $\text{Perv}_0$ are precisely those of the form $G[1]$ where $G$ is a constructible sheaf on $\bb{A}^1$ with vanishing cohomology.
\end{proposition}

\begin{theorem}[Artin's vanishing theorem {\cite[Corollaire 3.2]{Artin73}}]
    Let $F$ be a constructible sheaf on an affine variety $X$ of dimension $d$. Then $H^i(X, F) = 0$ if $i > d$.
\end{theorem}

We note the following definitions, which shall appear when we compare rapid decay cohomology to so-called \emph{perverse cohomology}.

\begin{definition}[Additive convolution and $\Pi$]
    For two complexes of constructible sheaves $A, B \in D^b_c(\bb{A}^1)$, their \emph{additive convolution} is given by
    \begin{equation}
        A * B := R\text{sum}_*(A \boxtimes B) \in D^b_c(\bb{A}^1).
    \end{equation}
    where $\text{sum}: \bb{A}^2 \to \bb{A}^1$ is the summation map. The endofunctor $\Pi: D^b_c(\bb{A}^1) \to D^b_c(\bb{A}^1)$ is defined by
    \begin{equation}
        A \to A * j_!j^{-1}\bb{C}[1]
    \end{equation}
    where $j: \bb{G}_m \hookrightarrow \bb{A}^1$.
\end{definition}

\begin{theorem}[{\cite[Prop. 2.4.3]{FJexp}}]
    \begin{itemize}
        \item For any $A \in \text{Perv}$, $\Pi(A) \in \text{Perv}_0$.
        \item $\Pi$ is a left-adjoint to the inclusion $\text{Perv}_0 \hookrightarrow \text{Perv}$.
        \item For any $A \in \text{Perv}$, the endofunctor $B \to A * B$ is exact with respect to the perverse $t$-structure. In particular, $\Pi$ is $t$-exact.
    \end{itemize}
\end{theorem}

\begin{definition}
    Let $e^{i\alpha}$ be a unit vector. The \emph{nearby fibre at infinity functor in the direction $e^{i\alpha}$} is the functor $\Psi_{e^{i\alpha}\infty}: \text{Perv}_0 \to \text{Vec}_\bb{C}$ given on objects by
\begin{equation}
    \Psi_{e^{i\alpha}\infty}(F[1]) = \underset{{r \to +\infty}}{\text{colim}} F(e^{i\alpha}\cdot S_r)
\end{equation}
where $S_r := \{z \in \bb{C}\vert\, \text{Re }z \geq r \}$.
\end{definition}
We record the following central result.
\begin{theorem}[{\cite[Thm. 12.6.6]{Katz90}}]
    The category $\text{Perv}_0$ is tannakian with product given by additive convolution. For any $\alpha \in \bb R$, a fibre functor is given by $\Psi_{e^{i\alpha}\infty}$.
\end{theorem}

For each $i \in \bb N_0$, Fres\'an and Jossen construct a perverse sheaf associated to a variety with potential $(X, f: X \to \bb A^1)$ from which the $i$th rapid decay cohomology can be recovered. The following argument does not appear in exactly this form in \cite{FJexp} but was communicated to the author at the seminar `Exponential Motives' in Oberwolfach, November 2024. \par
Take $X$ to be a smooth variety and $f: X \to \bb{A}^1$ a regular function, and let $\Delta \subset \bb{A}^1 \times \bb{A}^1$ denote the diagonal. Setting $\Gamma = \text{Graph}(f) \subset X \times \bb{A}^1$, we have $(f \times \id)^{-1}(\Delta) = \Gamma$ and thus
\begin{equation}
    (f \times \id)^{-1} \bb{Q}_{[\bb{A}^2, \Delta]} = \bb{Q}_{[X \times \bb{A}^1, \Gamma]}.
\end{equation}
Since $\bb A^1$ is contractible, the inclusion $\Gamma \hookrightarrow X \times \bb{A}^1$ is a homotopy equivalence and this sheaf has zero (global) cohomology. Note that it is also a constructible sheaf. Thus the associated spectral sequence of the composition of left-exact functors
\begin{equation}
    \modd{\bb{Q}}{X \times \bb{A}^1} \xrightarrow{(p_2)_*} \modd{\bb{Q}}{\bb{A}^1} \xrightarrow{\Gamma(\bb{A}^1, \cdot)} \text{Vec}_\bb{Q}
\end{equation}
converges to zero:
\begin{equation}
    E^{ij}_2 = R^i\Gamma(\bb{A}^1, R^j(p_2)_* \bb{Q}_{[X \times \bb{A}^1, \Gamma]}) \implies 0.
\end{equation}
However, since $R^j(p_2)_* \bb{Q}_{[X \times \bb{A}^1, \Gamma]}$ is constructible, by Artin's theorem, $E^{ij}_2 = 0$ if $i > 1$. Since the maps on the second page of the spectral sequence are of the type $E^{ij}_2 \to E^{i+2, j-1}$, all maps are zero and in particular, the sequence degenerates on the second page. Since it converges to zero, all of its entries are zero, and we have proven the following:
\begin{proposition} \label{Prop: exp perv const}
    $R^j(p_2)_* \bb{Q}_{[X \times \bb{A}^1, \Gamma]}$ has vanishing cohomology, and in particular by Prop. \ref{Prop: perv0 charac},
    \begin{equation}
        R^j(p_2)_* \bb{Q}_{[X \times \bb{A}^1, \Gamma]}[1] \in \text{Perv}_0.
    \end{equation}
\end{proposition}

\begin{remark}
    In \cite[Prop. 3.2.2]{FJexp}, the authors demonstrate that there is a canonical isomorphism
    \begin{equation}
        \Pi\left(\prescript{p}{}{\mcal{H}^j}(Rf_*\bb{Q})\right) \simeq R^j(p_2)_* \bb{Q}_{[X \times \bb{A}^1, \Gamma]}[1].
    \end{equation}
    The left-hand side is the so-called perverse cohomology. It is this result we prove a version of in Proposition \ref{Prop: preperv isom}, where we generalise to a local system $L$ on $\bb G_m$ with quasi-unipotent monodromy.
\end{remark}

\section{A vanishing result} \label{Sec: gen to other LS}

In this section, we prove the vanishing of certain cohomology groups, which shall be used in the results of Section \ref{Sec: comp to perv} on perverse sheaves and rapid decay cohomology (Proposition \ref{Prop: preperv isom}). The results in this section only depend of the structure of the variety $X$ as a topological space, so we can work over the category of locally compact topological spaces. \par

In the following, we use the notation $z^{\frac ab}\bb C$ to denote a rank 1 local system on $\bb G_m$ with monodromy $e^{2\pi i\frac ab}$. Recall also that if $X$ is a topological space or algebraic variety, $p_1: X \times \bb A^1 \to X$ and $p_2: X \times \bb A^1 \to \bb A^1$ denote the projections.

\begin{lemma} \label{Lem: GC vanishes 1}
    Let $X$ be a locally compact topological space, $f: X \to \bb{G}_m$ be a continuous map and $0 \leq a < b$ be integers. We set $\Gamma \subset X \times \bb{A}^1$ to be the graph of $X \xrightarrow{f} \bb{G}_m \hookrightarrow \bb{A}^1$. Then
    \begin{equation}
        R\Gamma\left(X \times \bb{A}^1, \left((f\circ p_1)^{-1}(z^{\frac{a}{b}}\bb{C})\right)_{[X\times\bb{A}^1, \Gamma]}\right) = 0.
    \end{equation}
\end{lemma}
\begin{proof}
    Let $r_b: \bb{G}_m \to \bb{G}_m$ be the map taking $z \to z^b$. We first note that
    \begin{equation}
        (r_b)_*\bb{C} \simeq \bb{C} \oplus z^{\frac{1}{b}}\bb{C} \oplus \dots \oplus z^{\frac{b-1}{b}}\bb{C}.
    \end{equation}
    Indeed, each fibre consists of $b$ points, so it a local system of rank $b$ and one sees easily that the monodromy matrix has characteristic polynomial $T^b - 1$. This means in particular that
    \begin{equation}
        (r_b \times \id)_*(r_b \times \id)^{-1}\bb{C} \simeq p_1'^{-1}\left(\bb{C} \oplus z^{\frac{1}{b}}\bb{C} \oplus \dots \oplus z^{\frac{b-1}{b}}\bb{C}\right).
    \end{equation}
    where $p'_1: \bb{G}_m \times \bb{A}^1 \to \bb{G}_m$ is the projection and $\id = \id_{\bb A^1}$. Setting $\Delta' \subset \bb{G}_m \times \bb{A}^1$ as the diagonal, we find via the projection formula that
    \begin{equation}
    \begin{aligned}
        (r_b \times \id)_*(r_b \times \id)^{-1}\left(\bb{C}_{[\bb G_m \times \bb A^1, \Delta']}\right) &\simeq p_1'^{-1}\left(\bb{C} \oplus \dots \oplus z^{\frac{b-1}{b}}\bb{C}\right)_{[\bb G_m \times \bb A^1, \Delta']} \\
        &\simeq p_1'^{-1}\left(\bb{C}\right)_{[\bb G_m \times \bb A^1, \Delta']} \oplus \dots \oplus p_1'^{-1}\left(z^{\frac{b-1}{b}}\bb{C}\right)_{[\bb G_m \times \bb A^1, \Delta']}.
    \end{aligned}
    \end{equation}
    Since $p'_1 \circ (f \times \id) = f \circ p_1$ and cohomology commutes with direct summations, it suffices to prove that $(f \times \id)^{-1}(r_b \times \id)_*(r_b \times \id)^{-1} \bb{C}_{[\bb{G}_m \times \bb{A}^1, \Delta']}$ has zero cohomology. \par

    \vspace{1pt}
    
    Choose $(Z, q_1: Z \to \bb{G}_m, q_2: Z \to X)$ so that the diagram \\
    \adjustbox{center}{
        \begin{tikzcd}[sep = large]
            Z \times \bb{A}^1 \arrow{r}{q_1 \times \id} \arrow{d}{q_2 \times \id} & \bb{G}_m \times \bb{A}^1 \arrow{d}{r_b \times \id} \\
            X \times \bb{A}^1 \arrow{r}{f \times \id} & \bb{G}_m \times \bb{A}^1
        \end{tikzcd}
    }
    is cartesian. The vertical maps are proper and all spaces involved are locally compact, so we have an isomorphism of sheaves on $X \times \bb{A}^1$
    \begin{equation}
    \begin{aligned}
        (f\times \id)^{-1}(r_b \times \id)_*[(r_b \times \id)^{-1}\bb{C}_{[\bb{G}_m \times \bb{A}^1, \Delta']}] &\simeq (q_2 \times \id)_*(q_1 \times \id)^{-1} [(r_b \times \id)^{-1}\bb{C}_{[\bb{G}_m \times \bb{A}^1, \Delta']}] \\
        &\simeq (q_2 \times \id)_* [((r_b \circ q_1) \times \id)^{-1}\bb{C}_{[\bb{G}_m \times \bb{A}^1, \Delta']}] \\
        &= (q_2 \times \id)_* [\bb{C}_{[Z \times \bb A^1, \Gamma']}]
    \end{aligned}
    \end{equation}
    where $\Gamma' \subset Z \times \bb{A}^1$ is the graph of $Z \xrightarrow{r_b \circ q_1} \bb{G}_m \hookrightarrow \bb{A}^1$. For the final equality, we have used that $\Gamma' = ((r_b \circ q_1) \times \id)^{-1}(\Delta')$ and applied Eqn. \eqref{eqn: [Z, A] pullback}. The map $q_2 \times \id$ is finite, in particular $R(q_2 \times \id)_* = (q_2 \times \id)_*$. Therefore the cohomology groups of $(q_2 \times \id)_*\bb{C}_{[Z\times \bb A^1, \Gamma']}$ are the same of those of $\bb{C}_{[Z\times \bb A^1, \Gamma']}$. But the inclusion $\Gamma' \hookrightarrow Z \times \bb{A}^1$ is a homotopy equivalence since $\bb{A}^1$ is contractible, and all cohomology groups vanish.
\end{proof}

\begin{lemma} \label{Lem: GC vanishes 2}
    Suppose $L$ is a local system of $\bb C$-vector spaces on $\bb{G}_m$ with quasi-unipotent monodromy. Let $f: X \to \bb{G}_m$, $\Gamma \subset X \times \bb{A}^1$ be as in Lemma \ref{Lem: GC vanishes 1} and $Y \subset X$ a closed subset. Then
    \begin{enumerate}
        \item $((f\circ p_1)^{-1}L)_{[X \times \bb{A}^1, Y \times \bb{A}^1 \cup \Gamma]}$ has zero cohomology, \label{GC van 2.1}
        \item $R^i(p_2)_*((f\circ p_1)^{-1}L)_{[X \times \bb{A}^1, Y \times \bb{A}^1 \cup \Gamma]}[1]$ is in $\text{Perv}_0$. \label{GC van 2.2}
    \end{enumerate}
\end{lemma}

\begin{proof}
    We proceed by induction on the rank of $L$, first assuming $Y$ is empty. The above lemma covers the rank $1$ case. Any eigenvector of the monodromy determines a rank one subsystem of $L$ and the quasi-unipotence means that its eigenvalue is a root of unity. Thus we have a subsheaf $z^{\frac{a}{b}}\bb{C} \hookrightarrow L$ for some $0\leq a < b$. Let $L'$ be its cokernel. The induction hypothesis applies to $L'$. Since $(f\circ p_1)^{-1}$ and $(\cdot)_{[X \times \bb{A}^1, \Gamma]}$ are exact functors, we get a short exact sequence 
    \begin{equation}
        0 \to ((f\circ p_1)^{-1}z^{\frac{a}{b}}\bb{C})_{[X \times \bb{A}^1, \Gamma]} \to ((f\circ p_1)^{-1}L)_{[X \times \bb{A}^1, \Gamma]} \to ((f\circ p_1)^{-1}L')_{[X \times \bb{A}^1, \Gamma]} \to 0
    \end{equation}
    Then from the long exact sequence in cohomology, the induction hypothesis and Lemma \ref{Lem: GC vanishes 1}, we find that the cohomology of the central term vanishes, proving \eqref{GC van 2.1} for $Y = \emptyset$. We derive the general case from the exact sequence
    \begin{equation}
        0 \to ((f\circ p_1)^{-1}L)_{[X \times \bb{A}^1, Y \times \bb{A}^1 \cup \Gamma]} \to ((f\circ p_1)^{-1}L)_{[X \times \bb{A}^1, \Gamma]} \to ((f\circ p_1)^{-1}L)_{[Y \times \bb{A}^1, \Gamma \cap Y \times \bb{A}^1]} \to 0.
    \end{equation}
    The cohomology of the last two terms vanishes (for the last term, it follows by applying the above to the pair $(Y, f\vert_Y)$) and therefore so does that of the first. Thus \eqref{GC van 2.1} is proven. For \eqref{GC van 2.2}, we use again the factorisation of functors
    \begin{equation}
        \Gamma(X \times \bb{A}^1, \cdot) = \Gamma(\bb{A}^1, \cdot) \circ (p_2)_*,
    \end{equation}
    the resulting spectral sequence and Artin's theorem (all exactly analogously to the proof of Prop. \ref{Prop: exp perv const}), we determine that $R^i(p_2)_*((f\circ p_1)^{-1}L)_{[X \times \bb{A}^1, \Gamma \cup Y \times \bb A^1]}$ has zero cohomology. Since it is constructible, shifting it one to the left in the derived category gives an object in $\text{Perv}_0$.
\end{proof}

\section{Comparison to perverse cohomology} \label{Sec: comp to perv}

Now we follow arguments in \cite[\S 3.2, 3.3]{FJexp}, adapted to our setting, that will lead us to a perverse sheaf associated to $(X, Y, f, \mcal{E})$ from which we can recover rapid decay cohomology, where $X$ is any variety over $k$, $Y$ a closed subvariety, $f: X \to \bb{G}_m$ and $\mcal{E} \in \modd{}{\D{\bb{G}_m}}$ satisfies $\qu_+$. The application is this paper is the basic lemma (Corollary \ref{Corr: basic lemma}). Recall that $p_1: X \times \bb{A}^1 \to X$ and $p_2: X \times \bb{A}^1 \to \bb{A}^1$ denote the projections.

\begin{proposition} \label{Prop: preperv isom}
    Let $g:= [X \xrightarrow{f} \bb{G}_m \hookrightarrow \bb{A}^1]$ and let $\mcal G$ be a sheaf on $X$ satisfying
    \begin{equation}
        R\Gamma(X \times \bb{A}^1, \left(p_1^{-1}\mcal G\right)_{[X \times \bb{A}^1, \Gamma\cup Y \times \bb{A}^1]}) = 0.
    \end{equation}
    Then
    \begin{equation}
        \Pi(Rg_*\mcal G_{[X, Y]}) \simeq R(p_2)_*\left(p_1^{-1}\mcal G\right)_{[X \times \bb{A}^1, \Gamma \cup Y \times \bb{A}^1]}[1]
    \end{equation}
\end{proposition}
\begin{proof}
    We have
    \begin{equation}
        [X \xhookrightarrow{\iota} \Gamma \cup Y \times \bb{A}^1 \xhookrightarrow{i} X \times \bb{A}^1 \xrightarrow{p_1} X] = \id_X
    \end{equation}
    and in particular,
    \begin{equation}
        \text{Hom}_X(\mcal G, \mcal G) = \text{Hom}_X(\mcal G, (p_1\circ i)_*\iota_*\mcal G) = \text{Hom}_{\Gamma \cup Y \times \bb{A}^1}((p_1\circ i)^{-1}\mcal G, \iota_*\mcal G)
    \end{equation}
    so there is a canonical morphism of sheaves on $\Gamma \cup Y \times \bb{A}^1$
    \begin{equation}
        i^{-1} p_1^{-1}\mcal G_{[\Gamma \cup Y \times \bb{A}^1, Y \times \bb{A}^1]} \to \iota_*\mcal G_{[X, Y]}.
    \end{equation}
    We apply $R(p_2 \circ i)_* = R(p_2)_* \circ i_*$ on both sides. It follows from $[X \xhookrightarrow{\iota} \Gamma \xrightarrow{p_2 \circ i} \bb{A}^1] = g$ that the resulting morphism is
    \begin{equation} \label{eqn: preperv1}
        R(p_2 \circ i)_*(i^{-1} p_1^{-1}\mcal G)_{[\Gamma \cup Y \times \bb{A}^1, Y \times \bb{A}^1]} = R(p_2)_*[i_*i^{-1}(p_1^{-1}\mcal G)_{[X \times \bb{A}^1, \Gamma \cup Y \times \bb{A}^1]})] \to Rg_* \mcal G_{[X, Y]}.
    \end{equation}
    This morphism is an isomorphism. Indeed, for a continuous map $h: S \to T$ of topological spaces and a sheaf $\mcal F$ on $S$, the direct image functor $R^ih_*\mcal F$ is the sheafification of the presheaf $U \to H^i(h^{-1}(U), \mcal F_U)$. We prove the isomorphism on the level of these presheaves, which of course carries over to the sheafification. On the right hand side, the presheaf evaluated on an open $U \subset \bb{A}^1$
    \begin{equation}
        \Gamma(U, (R^jg_* \mcal G_{[X, Y]})^{pre}) = H^i(g^{-1}(U), g^{-1}(U) \cap Y; \mcal G)
    \end{equation}
    for the while the left-hand side is
    \begin{equation}
    \begin{aligned}
        H^j(\Gamma_U \cup Y \times U, Y \times U; i^{-1} p_1^{-1}\mcal G) &\simeq H^j(\Gamma_U, \Gamma_U \cap Y \times U, i^{-1}p_1^{-1}\mcal G) \\
        &\simeq H^j(g^{-1}(U), g^{-1}(U) \cap Y, \mcal G)
    \end{aligned}
    \end{equation}
    where we have set $\Gamma_U := \Gamma \cap X \times U$ and used excision along with the identification $\Gamma_U \xrightarrow{\sim} g^{-1}(U)$. \par 
    We have the usual exact sequence
    \begin{equation}
        0 \to \left(p_1^{-1}\mcal G\right)_{[X \times \bb{A}^1, \Gamma \cup Y \times \bb{A}^1]} \to p_1^{-1}\mcal G \to i_*i^{-1}p_1^{-1}\mcal G \to 0
    \end{equation}
    inducing the exact triangle
    \begin{equation}
        R(p_2)_*\left(p_1^{-1}\mcal G\right)_{[X \times \bb{A}^1, \Gamma \cup Y \times \bb{A}^1]} \to R(p_2)_*p_1^{-1}\mcal G \to R(p_2)_*i_*i^{-1}p_1^{-1}\mcal G \xrightarrow{+1}
    \end{equation}
    The middle term is the constant sheaf $R\Gamma(X, \mcal G)$ on $\bb{A}^1$. It follows easily from its definition that $\Pi$ annihilates constant sheaves, leaving an isomorphism
    \begin{equation}
        \Pi \left(Rg_*\mcal G_{[X, Y]}\right) \simeq \Pi\left(R(p_2)_*\left(p_1^{-1}\mcal G\right)_{[X \times \bb{A}^1, \Gamma \cup Y \times \bb A^1]}\right)[1].
    \end{equation}
    where we have also used the isomorphism \eqref{eqn: preperv1}. It remains only to show that the canonical map
    \begin{equation}
        R(p_2)_*\left(p_1^{-1}\mcal G\right)_{[X \times \bb{A}^1, \Gamma \cup Y \times \bb{A}^1]} \to \Pi\left(R(p_2)_*\left(p_1^{-1}\mcal G\right)_{[X \times \bb{A}^1, \Gamma\cup Y \times \bb{A}^1]}\right)
    \end{equation}
    is an isomorphism. We use the exact triangle from equation (2.4.3.2) in \cite{FJexp}
    \begin{equation}
    \begin{gathered}
        R\Gamma(\bb{A}^1, R(p_2)_*\left(p_1^{-1}\mcal G\right)_{[X \times \bb{A}^1, \Gamma\cup Y \times \bb{A}^1]}) \to R(p_2)_*\left(p_1^{-1}\mcal G\right)_{[X \times \bb{A}^1, \Gamma\cup Y \times \bb{A}^1]} \\ \to \Pi\left(R(p_2)_*\left(p_1^{-1}\mcal G\right)_{[X \times \bb{A}^1, \Gamma\cup Y \times \bb{A}^1]}\right) \xrightarrow{+1}
    \end{gathered}
    \end{equation}
    where the first term is interpreted as a constant complex of sheaves on $\bb{A}^1$. It vanishes by $R\Gamma(\bb{A}^1, \cdot) \circ R(p_2)_* \simeq R\Gamma(X \times \bb{A}^1, \cdot)$ and the assumption on $\mcal G$, concluding the proof.
\end{proof}

\begin{corollary} \label{Corr: preperv isom}
    There is a canonical isomorphism
    \begin{equation}
        \Pi(Rg_*\mcal (f^{-1}L)_{[X, Y]}) \simeq R(p_2)_*\left((f\circ p_1)^{-1} L\right)_{[X \times \bb{A}^1, \Gamma \cup Y \times \bb{A}^1]}[1].
    \end{equation}
\end{corollary}
\begin{proof}
    This follows immediately from Lemma \ref{Lem: GC vanishes 2}\eqref{GC van 2.1} and Prop. \ref{Prop: preperv isom} with $\mcal G = f^{-1}L$.
\end{proof}

\begin{definition}
    The \emph{perverse cohomology} of $(X, Y, f^*\mcal{E})$ is
    \begin{equation}
        H^i_{perv}(X, Y, f^*\mcal{E}) := \Pi(\prescript{p}{}{\mcal{H}}^i(Rg_*(f^{-1}L)_{[X, Y]}))
    \end{equation}
\end{definition}

\begin{corollary}
    Suppose $\mcal E$ satisfies Assumption \ref{Assump: half plane} for some $\alpha \in \bb R/2\pi \bb Z$. Then there is a canonical isomorphism
    \begin{equation}
        \Psi_{e^{i\alpha}\infty}(H_{perv}^i(X, Y, f^*\mcal{E})) \simeq H^i_{rd}(X, Y, f^*\mcal{E})
    \end{equation}
\end{corollary}
\begin{proof}
    We follow \cite[Corr. 3.2.3]{FJexp} exactly. For any object $C$ in the derived category of constructible sheaves, we have isomorphisms in $\text{Perv}_0$
    \begin{equation}
        \Pi(\prescript{p}{}{\mcal{H}}^i(C)) \simeq \prescript{p}{}{\mcal{H}}^i(\Pi(C)) \simeq \mcal{H}^{i-1}(\Pi(C))[1].
    \end{equation}
    From Corollary \ref{Corr: preperv isom}, setting $C = Rg_*(f^{-1}L)_{[X, Y]}$, we find
    \begin{equation}
    \begin{aligned}
        H^i_{perv}(X, Y, f^*\mcal{E}) &= \mcal H^{i-1}(\Pi(Rg_*(f^{-1}L)_{[X, Y]})[1] \\
        &=  \mcal{H}^{i-1}(R(p_2)_*((f \circ p_1)^{-1}L)_{[X \times \bb{A}^1, \Gamma \cup Y \times \bb{A}^1]}[1])[1] \\
        &= R^i(p_2)_*((f \circ p_1)^{-1}L)_{[X \times \bb{A}^1, \Gamma \cup Y \times \bb{A}^1]}[1].
    \end{aligned}        
    \end{equation}
    The result follows from Corr. \ref{Corr: RDC isom FF}.
\end{proof}

\subsection{A basic lemma}

Just as in the exponential case, the perverse realisation provides a quick proof of a basic lemma, with the help of Beilison's famous result: 

\begin{theorem}[Beilinson's basic lemma]
    Let $f: X \to S$ be a morphism of quasi-projective varieties and let $F$ be a perverse sheaf on $X$. Set $d = \dim X$. Then there exists a dense open subset $j: U \hookrightarrow X$ such that
    \begin{equation}
        \prescript{p}{}{\mcal{H}}^i(Rf_*j_!j^{-1} F) = 0 \text{ for all } i < 0.
    \end{equation}
\end{theorem}

\begin{proposition}[Basic lemma] \label{Corr: basic lemma}
    Let $X$ be an affine variety of dimension $d$, $Y \subset X$ a closed subvariety and let $f: X \to \bb{G}_m$ be a morphism. We take $\mcal{E} \in \modd{}{\D{\bb{G}_m}}$ satisfying Assumption \ref{Assump on E}. Then there exists a subvariety $Y \subset Z \subset X$ such that
    \begin{equation}
        H^i_{rd}(X, Z, f^*\mcal{E}) = 0 \text{ for all } i \neq d,
    \end{equation}
    where $H^i_{rd}$ is meant in the sense of Defn. \ref{Defn: sing RDC} if $X$ is singular.
\end{proposition}
\begin{proof}
    The proof is exactly the same as the proof of the exponential basic lemma \cite[Corollary 3.3.3]{FJexp}. The only thing to be noted is that $f^{-1}L$ is a local system on a smooth variety of dimension $d$, so $f^{-1}L[d] \in D^b_c(X)$ is perverse.
\end{proof}

\section{Endowing RD cohomology with a rational structure} \label{Sec: rational structures}

Let $K \subset \bb{C}$ be a subring of the complex numbers. In the following, we assume for ease of notation that $Y = \emptyset$.
\begin{definition}
    A $K$-rational structure on a local system of $\bb{C}$-vector spaces $L$ is a local system of $K$-modules $L_K$ together with an isomorphism of complex local systems $L_K \otimes_K \bb{C} \xrightarrow{\sim} L$.
\end{definition}
We can now state the main result of this paper.
\begin{theorem} \label{Thm: main result}
    Let $X$ be a (smooth) variety over $k$ and let $Y \subset X$ be a closed subvariety. We take $\mcal E \in \modd{}{\D{\bb G_m}}$ satisfying $\mbf{A}_+$ (Assumptions \ref{Assump on E} and \ref{Assump: half plane}) and let $L$ be its corresponding local system of horizontal sections. Then any $K$-rational structure $L_K$ on $L$ induces a $K$-rational structure on the rapid decay cohomology, that is, for each $i \geq 0$, there is a $K$-vector space $H^i_{rd}(X, f^*\mcal{E})_{K}$ equipped with a canonical isomorphism
    \begin{equation}
        H^i_{rd}(X, Y, f^*\mcal{E}) = H^i_{rd}(X, Y, f^*\mcal{E})_{K} \otimes_{K} \bb{C}.
    \end{equation}
\end{theorem}
\begin{proof}
    This follows straightforwardly from the work we have already done. Keeping the notation of Section \ref{sec: Conn exp type}, Thm. \ref{Thm: L* calcs RDC} states that if $\mcal E$ satisfies Assumption $\mbf{A}_+$ for some $\sigma \in \partial\rob$, there is a canonical isomorphism 
\begin{equation} \label{eqn: RDC comp isom}
    H^i_{rd}(X, Y, f^*\mcal{E}) \simeq H^i\left(\widetilde{X}, \widetilde Y \cup \tilde f^{-1}(\sigma), \tilde f^{-1}L\right).
\end{equation}
Given a $K$-rational structure $L_K$ on $L$, the isomorphism $L \simeq L_K \otimes_K \bb{C}$ means that
\begin{equation}
    H^i_{rd}(X, Y, f^*\mcal{E}) \simeq H^i\left(\widetilde{X}, \widetilde Y \cup \tilde f^{-1}(\sigma), \tilde f^{-1}L_K \otimes \bb C\right) \simeq H^i\left(\widetilde{X}, \widetilde Y \cup \tilde f^{-1}(\sigma), \tilde f^{-1}L_K\right) \otimes_K \bb{C}
\end{equation}
inherits a $K$-rational structure.
\end{proof}

\subsection{Determining the $\bb Q(\zeta)$-structure} \label{Sec: determining structure}

Recall that any connection on $\bb G_m$ of type $E$ has a basis of solutions of the form
\begin{equation} \label{eqn: andre basis 2}
    \begin{pmatrix}
        F_1 & \dots & F_n
    \end{pmatrix} z^{\Gamma_0},
\end{equation}
with $F_i$ an $E$-function and $\Gamma_0$ an upper triangular matrix with entries in $\bb Q$. More explicitly, assuming $\Gamma_0$ to be in Jordan normal form and restricting attention to the top Jordan block of size $m >0$ and eigenvalue $q \in \bb Q$, the solutions are of the form
\begin{equation}
    z^q\cdot \left(F_1, F_2+F_1\log(z), F_3 + F_2\log(z) + F_1 \frac{\log^2(z)}{2!}, \dots, F_m + F_{m-1}\log(z)+\dots+F_1\frac{\log^m(z)}{m!}\right),
\end{equation}
Note that each Jordan block corresponds to a subsystem of $L = \text{Sol}(\mcal E)$. Since $\log(e^{2\pi i}z) = 2\pi i+\log(z)$, with respect to this basis, the monodromy has entries in $\bb Q(e^{2\pi iq})[2\pi i]$. It is easy to check that the basis
\begin{equation}
\begin{aligned}
z^q\cdot \Bigg((2\pi i)^{m-1}F_1, (2\pi i)^{m-2}(F_2+ F_1\log(z)), (2\pi i)^{m-3}\left(F_3 + F_2\log(z) + F_1 \frac{\log^2(z)}{2!}\right), \\
\dots, (2\pi i)^0\left(F_m + F_{m-1}\log(z)+\dots+F_1\frac{\log^m(z)}{m!}\right)\Bigg).    
\end{aligned}
\end{equation}
has monodromy defined over $\bb Q(e^{2\pi iq})$. Choose $\zeta$ to be a root of unity of order equal to the lowest common multiple of the denominators of the eigenvalues of $\Gamma_0$. Doing the same for all Jordan blocks, the local system $L$ is endowed with a $\bb Q(\zeta)$-structure. In the setting of Thm. \ref{Thm: main result}, this induces a $\bb Q(\zeta)$-structure on the rapid decay cohomology groups, which gives us the version stated in Thm. \ref{Thm: intro}.\par

\begin{remark} \label{Rem: non canon}
    This structure cannot be chosen canonically; the choice of basis is key. Having chosen $F_j$ for $j < i$, each $F_i$ is determined only up to addition of some $k^\times$-multiple of $F_1$. However, one may check that the sheaf of $k(\zeta)[2\pi i]$-modules generated by a basis of the form \eqref{eqn: andre basis 2} is independent of the choice, and the induced $k(\zeta)[2 \pi i]$-structure on the cohomology groups is thus canonical.
\end{remark}

\begin{remark}
    Assumption \ref{Assump: half plane} is quite natural. Our principal interest is to investigate absolutely convergent integrals of the form
    \begin{equation}
        \int_\sigma F\circ f(x)\cdot \omega
    \end{equation}
    where $X$ is a variety, $f: X \to \bb G_m$ is a non-vanishing function, $F$ is some $E$-function, $\sigma$ is a rapid decay cycle for $F\circ f(x)$ and $\omega$ is a regular differential form on $X$. Choose $\mcal E$ with $F$ as a horizontal section. Then period pairing on the cohomology groups of $(X, f, \mcal E)$ may be considered in this way, that is, in terms of integration of rapid-decay cycles (see Section \ref{Subsec: RDC}). If $F$ has no rapid decay direction at $\infty$, the image of $\sigma$ is not permitted to approach the boundary of $X$. We can only expect $F$ to have a rapid decay direction if Assumption \ref{Assump: half plane} is satisfied by $\mcal E$.
\end{remark}

Furthermore, it is very easy to construct such examples. Given any $\D{\bb G_m}$-module $\mcal E$ of type $E$, the `shifted' module $\mcal E \otimes \mcal{O}_{\bb G_m}e^{\alpha z}$ satisfies $\qu_+$ for $\abs\alpha$ large enough (the exponents are $\lambda_i+\alpha$).

With this in mind, we make the following tentative definition.
\begin{definition}[$E$-period] \label{Defn: E period}
    Let $\Psi$ be an $E$-operator defined over a number field $k$ and let $\mcal E = \D{\bb A^1}/\D{\bb A^1}\Psi \in \modd{}{\D{\bb A^1}}$. Let $X$ be a smooth variety, $Y$ a closed subvariety and $f: X \to \bb A^1$, all defined over $k$, be such that one of the following holds:
    \begin{enumerate}
        \item $f(X)$ is finite and does not contain $0 \in \bb A^1$; 
        \item $\mcal{E}$ satisfies Assumption \ref{Assump: half plane} and $0 \in f(X)$ only if $0$ is not a singularity of $\Psi$.
    \end{enumerate}
    Choose a $\bb Q(\zeta)$-structure on $L = \text{Sol}(\mcal E)$ as described above for some root of unity $\zeta$. We define an \emph{effective $E$-period} of the quadruple $(X, Y, f, \mcal E)$ to be an element of the $\bb Q(\zeta) \cdot k$-vector space generated by the matrix elements of the comparison isomorphism
    \begin{equation}
        H^i_{dR}(X, Y, f^*\mcal{E}) \otimes_k \bb{C} \xrightarrow{\sim} H^i_{rd}(X, Y, f^*\mcal{E})_{\bb Q(\zeta)} \otimes_{\bb Q(\zeta)} \bb{C}.
    \end{equation}
    We define the \emph{ring of effective $E$-periods} $\mcal P_{\text{eff}}^E$ to be the sub-$\bb Q$-algebra of $\bb C$ generated by the $E$-periods of all quadruples as above and ranging over all number fields $k$. We define the \emph{ring of $E$-periods} to be $\mcal P^E = \mcal P^E_{\text{eff}}[(2\pi i)^{-1}]$.
\end{definition}

\begin{remark}
    Formulated in this way, the exponential periods form a subset of the $E$-periods. By a note of Rivoal and Fischler \cite{Rivoal17}, there are essentially no interesting $E$-operators without a singularity at $0$ besides the usual exponential one. Thus, we have not lost anything by assuming all our maps $f$ to land in $\bb G_m$ up to this point.
\end{remark}

\section{Conclusion and outlook} \label{Sec: Conc}

The broad goal of this paper is to define a generalisation of the exponential periods by considering a wider class of algebraic $\D{}$-modules than the exponential twists of the structure sheaf. Our investigation has shone some light on the circumstances under which one can reasonably expect to be able to perform such a construction. In this setting, our version of rapid decay cohomology on singular varieties and Nori's basic lemma could also be key tools in a motivic theory of these numbers.

An \emph{$E$-value} is a complex number realised as the evaluation of an $E$-function at an algebraic argument $\alpha \in \qbar^\times$. Since $F(z)$ is an $E$-function if and only if $F(\alpha z)$ is, we may take this argument to be $1$. As $E$-functions themselves form a ring, it is manifest that $E$-values do so as well. They have very interesting arithmetic properties, most notably the theorem of Siegel-Shidlovskii \cite{NestShid96, andre00_2, Beukers06}. \par
Let $F$ be an $E$-function and $\alpha \in k^\times$. Then $F$ is a horizontal section of some $\mcal E \in \modd{}{\D{\bb G_m}}$ of type $E$. Choosing $X = \spec k$ and $f: \spec k \to \bb G_m$ to be the map with image $\alpha$ realises $(2\pi i)^nF(\alpha)$ as an effective $E$-period for some $n \geq 0$, so we have the inclusion
\begin{equation}
    \{E\text{-values}\} \subset \mcal P^E.
\end{equation}

Another aim of this paper is to attempt to lay some foundations for a study of (the ring of) $E$-values and exponential periods, under the wider umbrella of $E$-periods. It is as yet unclear if $\mcal P^E$ constitutes a significant enlargement of either ring. It is well-known that certain $E$-values are exponential periods, for example special values of the Bessel functions and of the exponential integral function $E(z) = \int_1^\infty e^{-zx}\frac{\diff x}x$. Furthermore, in another upcoming work of Fres\'an and Jossen on $E$-operators arising from geometry, it is shown that the exponential periods form a subring of the ring
\begin{equation*}
    R := \bb Q\left[\{\text{periods}\}, \{\Gamma(a)\}_{a \in \bb Q\cap (0, 1)}, \gamma, \{E\text{-values}\}\right],
\end{equation*}
where $\gamma$ denotes the Euler constant. This result has already appeared in the notes \cite{Fresannotes}. Since $\gamma$ and $\Gamma(a)$ ($a \in \bb Q\cap (0, 1)$) are exponential periods and all exponential periods are $E$-periods, it follows that $R \subset \mcal P^E$. We ask if this inclusion is in fact an equality, which would mean a kind of `closure' result, i.e. up to multiplication by classical periods, rational Gamma values and the Euler constant, there are no interesting numbers obtained by integrating $E$-functions besides their own special values. This is a topic for further research and one we expect to return to. \par

\nocite{*}
\bibliographystyle{alpha}
\bibliography{Bibliography}

\end{document}